\documentclass[a4paper,12pt]{amsart}
\usepackage{amssymb,amstext,amsmath,amscd,amsthm,
amsfonts,enumerate,graphicx,latexsym}
\usepackage{url}
\usepackage[usenames]{color}
\usepackage{multirow}
\definecolor{gr}{rgb}{0.1, .5 , .10}
\usepackage[all]{xy}

\newtheorem{theorem}{Theorem}[section]
\newtheorem{theorem*}{Theorem}
\newtheorem{corollary}[theorem]{Corollary}
\newtheorem{corollary*}[theorem*]{Corollary}
\newtheorem{lemma}[theorem]{Lemma}
\newtheorem{proposition}[theorem]{Proposition}

\newtheorem*{proposition*}{Proposition}
\theoremstyle{definition}

\newtheorem{remark}[theorem]{Remark}
\newtheorem*{remark*}{Remark}
\newtheorem{question}[theorem]{Question}
\newtheorem*{question*}{Question}

\newtheorem*{conjecture*}{Conjecture}
\newtheorem{example}[theorem]{Example}
\newtheorem*{example*}{Example}

\newtheorem*{notation*}{Notation}

\newtheorem*{claim*}{Claim}

\newtheorem*{notationremark*}{Notation and Remark}
\numberwithin{equation}{theorem}
\def\Ext{\operatorname{Ext}}

\def\rad{\operatorname{rad}}
\def\soc{\operatorname{soc}}

\def\End{\operatorname{End}}

\def\Ltensor{\otimes^\mathsf{L}}

\def\Db{\mathsf{D}^{\rm b}}

\def\D{\operatorname{\mathsf{D}}}

\def\mod{\operatorname{\mathsf{mod}}}
\def\smod{\operatorname{\underline{\mathsf{mod}}}}
\def\proj{\operatorname{\mathsf{proj}}}

\def\Db{\mathsf{D^b}}
\def\Kb{\mathsf{K^b}}

\def\silt{\operatorname{\mathsf{silt}}}
\def\tilt{\operatorname{\mathsf{tilt}}}

\newcommand{\dsilt}[2]{\operatorname{\mathsf{#1_{\mbox{{\scriptsize $#2$}}}-silt}}}
\def\AA{\mathbb{A}}
\def\B{\mathcal{B}}
\def\C{\mathcal{C}}

\def\D{\mathbb{D}}

\def\EE{\mathbb{E}}

\def\G{\mathsf{G}}
\def\H{\mathcal{H}}
\def\II{\mathcal{I}}

\def\RR{\mathcal{R}}

\def\Z{\mathbb{Z}}
\newcommand{\op}{{\rm op}}

\newcommand{\tr}{{\rm tr}}

\newcommand{\dbracket}[1]{[\! [ #1 ] \! ]}

\newcommand{\old}[1]{{\color{red} #1}}

\begin{document}
%\allowdisplaybreaks
\setlength{\baselineskip}{15pt}
\title{Examples of tilting-discrete symmetric algebras}
\author{Takuma Aihara}
\address{Department of Mathematics, Tokyo Gakugei University, 4-1-1 Nukuikita-machi, Koganei, Tokyo 184-8501, Japan}
\email{aihara@u-gakugei.ac.jp}
%\author{}
%\address{}
%\email{}

\thanks{TA was partly supported by JSPS KAKENHI Grant Number JP25K06950.}
\begin{abstract}
We give several examples of tilting-discrete symmetric algebras; in particular, one explores which algebra has tilting-discrete trivial extension.
We provide a counter example of the conjecture stating any $\tau$-tilting finite symmetric algebra is tilting-discrete.
Also, we discuss the tilting-disconnectedness of symmetric algebras and give new examples of tilting-disconnected symmetric algebras.
\end{abstract}
\maketitle
%\tableofcontents
%%%%%%%%%%%%%%%%%%%%%%%%%%%%%%%%%%%%%%%%%%%%%%%%%%%%%%%%
\section{Introduction}

The tilting-discreteness of a (finite dimensional) symmetric algebra is a finiteness condition on tilting complexes; that is, a tilting-discrete symmetric algebra has only finitely many tilting complexes of each length \cite{Ai}.
In the case, we can `essentially' obtain all tilting complexes via tilting mutation; 
roughly speaking, such a symmetric algebra is said to be tilting-connected.
Thus, we can grasp the whole picture of tilting complexes over a tilting-discrete symmetric algebra.
Typical examples of tilting-discrete symmetric algebras are representation-finite symmetric algebras \cite{Ai}, Brauer graph algebras whose Brauer graphs admit at most one cycle of odd length and none of even length \cite{AAC}, symmetric algebras of dihedral, semidihedral and quaternion type \cite{EJR} and symmetric algebras of tubular type with nonsingular Cartan matrix \cite{AHMW}.
A big problem is to classify all tilting-discrete symmetric algebras.

To reach the goal, $\tau$-tilting finiteness is a must-check property, because it is a necessary condition for a symmetric algebera to be tilting-discrete \cite{AIR, DIJ, DIRRT}; 
in particular, the endomorphism algebras of tilting complexes over a given symmetric algebra are all $\tau$-tilting finite if and only if the symmetric algebra is tilting-disctete \cite{Ai}.
Note also that the notion of support $\tau$-tilting modules coincides with that of 2-term tilting complexes for a symmetric algebra.
Then, we are curious about the following folklore conjecture:

\begin{conjecture*}
Any $\tau$-tilting finite symmetric algebra is tilting-discrete.
\end{conjecture*}

First, we give an easy observation to judge the truth of the conjecture; a \emph{3-term} complex is defined to be a complex whose $i$-th term is zero unless $i=0,-1,-2$.

\begin{proposition*}
\begin{enumerate}
\item If every $\tau$-tilting finite symmetric algebra $\Lambda$ admits only finitely many 3-term (basic) tilting complexes, then the conjecture holds true.
\item The following are equivalent:
\begin{enumerate}
\item Any $\tau$-tilting finite symmetric algebra is tilting-discrete;
\item For every (not necessarily, symmetric) algebra $A$, the trivial extension of $A$ is tilting-discrete if and only if it is $\tau$-tilting finite;
%\item For any symmetric algebra $A$, the trivial extension of $A$ is tilting-discrete if and only if it is $\tau$-tilting finite.
\end{enumerate}
\end{enumerate}
\end{proposition*}
\begin{proof}
(1) Due to \cite[Theorem 2.4]{AM}, we have only to show that for any irreducible tilting mutation $T$ of $\Lambda$, there exist only finitely many tilting complexes $U$ with $T\geq U\geq T[1]$.
Evidently, we have $\Lambda\geq T\geq U\geq T[1]\geq \Lambda[2]$; hence we are done by the assumption.

(2) We need only check the implication (ii)$\Rightarrow$(i).
Let $A$ be a $\tau$-tilting finite symmetric algebra.
The symmetry of $A$ indicates that the trivial extension $\Lambda$ of $A$ is isomorphic to $A\otimes_KK[x]/(x^2)$.
By \cite[Theorem 2.1 (Corollary 2.2)]{AH} and its proof, we see that $\Lambda$ is $\tau$-tilting finite, which implies by assumption that $\Lambda$ is tilting-discrete.
Again, by \cite[Theorem 2.1]{AH} and its proof the functor $-\otimes_KK[x]/(x^2): \Kb(\proj A)\to\Kb(\proj\Lambda)$ induces an injection $\tilt A\to \tilt\Lambda$ that preserves the partial order.
This completes the proof.
%we conclude that $A$ is tilting-discrete.
\end{proof}

Thus, we focus on trivial extension algebras to grope for answers of the conjecture.
An advantage to do that is the following.
Let $A$ be a (not necessarily, symmetric) algebra with Cartan matrix $C$.
If the trivial extension of $A$ is $\tau$-tilting finite, then we have:
\begin{enumerate}
\item  \cite[Theorem 5.12(d)]{DIRRT} $A$ is also $\tau$-tilting finite;
\item \cite[Proposition 3.2]{H} The quadratic form $\overline{x^\tr}Cx$ has a strictly positive real part for complex vectors $x$, and equivalently, $C$ is invertible and the (usual) homological quadratic form $x^\tr C^{-\tr}x$ (i.e., Euler form) for real vectors $x$ is positive definite.
Here, $(-)^\tr$ and $\overline{(-)}$ denote the transpose and the complex conjugate of a matrix.
\end{enumerate}
We refer to the condition (2) as a \emph{positive definite Cartan matrix}.
In this case, the algebra $A$ satisfies certain good properties; see Appendix \ref{append:cyclotomic}.
Note that the conditions (1) and (2) are independent (Example \ref{TauTFvsPDC}).
We also remark that even if every algebra with the same trivial extension is $\tau$-tilting finite and has a positive definite Cartan matrix, its trivial extension is not necessarily $\tau$-tilting finite  (Example \ref{withsameTE}).

Then, we find a crucial example of algebras.
Let $\ell, m\geq1$ and $Q_m$ be the quiver 
\[\xymatrix{
1 \ar@(lu,ld)_x\ar@{=>}[r]^y & 2
}\]
Here, the double arrow stands for $\ell$ arrows being drawn from 1 to 2 which are all labelled by $y$; note that these arrows are actually different, but we write them in this manner without confusion.
Then, we set $A_m:=KQ_m/(x^m, xy)$.
One remarks that $A_1$ for any $\ell$ is nothing but the $\ell$-Kronecker algebra.
Here is the first main theorem of this paper, which answers the conjecture negatively and gives a complete list of the trivial extensions of the $\ell$-Kronecker algebras to be tilting-discrete, tilting-connected and $\tau$-tilting finite.

\begin{theorem*}[Example \ref{Am} and Theorem \ref{Kronecker}]
Let $\Lambda$ be the trivial extension of $A_m$.
%Then we have the following:
\begin{enumerate}
\item Let $m>1$.
\begin{enumerate}
\item $\Lambda$ is never tilting-discrete;
\item It is $\tau$-tilting finite if and only if $\ell=1$.
\end{enumerate}
\item Let $m=1$.
\begin{enumerate}
\item If $\ell=1$, then $\Lambda$ is tilting-discrete;
\item Otherwise, it is neither $\tau$-tilting finite nor tilting-connected.
\end{enumerate}
\end{enumerate}
\end{theorem*}

We will also discuss the silting-discreteness of $A_m$ (Proposition \ref{AmSID}); note that $A_m$ for $\ell=1$ is always $\tau$-tilting finite (more strongly, it is representation-finite.)

\begin{remark*}
As we solved the conjecture negatively, one sees that (2) of the proposition fails.
On the other hand, we restore (1) of the proposition as follows.
Let us consider a property (P) of symmetric algebras such that 
\begin{enumerate}[(i)]
\item it is a derived invariant;
\item it ($+$ the $\tau$-tilting finiteness) implies the finiteness of 3-term (basic) tilting complexes.
\end{enumerate} 
Then any $\tau$-tilting finite symmetric algebra satisfying the property (P) is tilting-discrete.

We remark by Proposition \ref{2vs3} that the condition (ii) can be replaced by the following:
\begin{enumerate}[(i)]
\setcounter{enumi}{2}
\item the property (P) ($+$ the $\tau$-tilting finiteness) yields that every 2-term tilting complex has $\tau$-tilting finite endomorphism algebra.
\end{enumerate}
For example, the representation-finiteness of symmetric algebras can be used as (P).
\end{remark*}

A next problem is when the $\tau$-tilting finiteness of a symmetric algebra can be lifted to the tilting-discreteness (\emph{the lifting problem}).
A model case is Brauer graph algebras \cite{AAC} (Theorem \ref{sdBGA}).
In this context, we determine gentle algebras whose trivial extensions are tilting-discrete; they are Brauer graph algebras with multiplicity identically one \cite{Sc}.

\begin{theorem*}[Theorem \ref{deriveddiscrete}]
The following are equivalent for a gentle algebra $A$:
\begin{enumerate}
%\item The trivial extension of $A$ is a Brauer graph algebra whose Brauer graph has at most one cycle of odd length and none of even length;
\item The trivial extension of $A$ is tilting-discrete;
\item It is $\tau$-tilting finite;
\item $A$ has a positive definite Cartan matrix;
\item It is a derived-discrete algebra with nonsingular Cartan matrix.
\end{enumerate}
\end{theorem*}

We also classify tilting-discrete trivial extensions of clucter-tilted algebras.

\begin{theorem*}[Theorem \ref{clustertilted}]
Let $A$ be a nonhereditary cluster-tilted algebra of type $Q$. Then the following conditions are equivalent:
\begin{enumerate}
\item The trivial extension of $A$ is tilting-discrete;
\item It is $\tau$-tilting finite;
\item $Q$ is of Dynkin type $\AA$ and $A$ has precisely one oriented 3-cycle.
\end{enumerate}
\end{theorem*}

For more examples, we consider a variation of $A_m$ above, but their trivial extensions behave quite differently.
Let $A_m^\circ$ be the algebra presented by the quiver $Q_m$ with relation $x^m=0$.
Then we have the following result.

\begin{theorem*}[Proposition \ref{TEofAm2}]
Let $\Lambda$ be the trivial extension of $A_m^\circ$ and assume $\ell=1$.
Then the following conditions are equivalent:
\begin{enumerate}
\item $\Lambda$ is tilting-discrete;
\item It is $\tau$-tilting finite;
\item It has a positive definite Cartan matrix;
\item It is tilting-connected;
\item $m\leq3$.
\end{enumerate}
\end{theorem*}

We will establish the silting-discreteness and the $\tau$-tilting finiteness of $A_m^\circ$ for $\ell=1$ as well (Proposition \ref{Am2}).

For the lifting problem, we saw that Cartan matrices play an important role.
Moreover, it is significant to study derived equivalent algebras to a given symmetric algebra for tilting-discreteness.
Thus, we have an eye on certain values; let $X$ be an indecomposable partial tilting complex (i.e., a direct summand of some tilting complex) and $\delta(X)$ denote the dimension of the endomorphism algebra of $X$ (in the perfect derived category).
In other words, this value coincides with a diagonal entry of the Cartan matrix of a derived equivalent algebra to a given (symmetric) algebra.
Improving a result of Eisele--Janssens--Raedschelders slightly \cite[Theorem 13]{EJR} (Theorem \ref{upperbound}), we have the following result.

\begin{theorem*}[Corollary \ref{upperbound2}]
Let $\Lambda$ be a symmetric algebra and assume that there is an upper bound for the valules $\delta(X)$ of indecomposable partial tilting complexes $X$.
If $\Lambda$ is $\tau$-tilting finite, then it is tilting-discrete.
\end{theorem*}

Also, we explore the tilting-disconnectedness of symmetric algebras; for example,

\begin{theorem*}[Theorem \ref{tdcBGA}]
Let $G$ be a Brauer graph which is either of the following:
\begin{enumerate}[{\rm (i)}]
\item it has only one vertex and at least two edges, or;
\item it has precisely two vertices, at least two edges and is bipartite.
\end{enumerate}
Then the Brauer graph algebra of $G$ is not tilting-connected.
\end{theorem*}

\begin{notationremark*}
Throughout this paper, all algebras mean a finite dimensional algebra over an algebraically closed filed $K$ which is indecomposable and basic, unless otherwise notified.
When an algebra $A$ is given, we fix a complete set $\{e_1, \cdots, e_n\}$ of orthogonal primitive idempotents and put $P_i:=e_iA$;
in particular, if $A$ is presented by a quiver with relations, then we take such an $e_i$ as the primitive idempotent corresponding to each vertex.
All modules are right and finite dimensional.

We denote by $\Db(\mod-)$ and $\Kb(\proj-)$ the bounded derived category and the perfect derived category, respectively.
The Grothendieck group is denoted by $K_0(-)$.

We freely use elementary facts on silting/tilting complexes and mutations from \cite{Ai, AI} and refer to \cite{Sc} for Brauer graph algebras.
In this paper, we will often deal with symmetric algebras, and then we do not distinguish silting/tilting theory; see \cite{AI}.
\end{notationremark*}

\section{Brauer graph algebras}

%Let $G$ be a Brauer graph with $n$ vertices.
%and multiplicity identically one; since tilting theory for Brauer graph algebras essentially does not depend on the multiplicities of vertices (cf. \cite{AAC}), we may always suppose that Brauer graphs admit multiplicity identically one.
%Note that each edge $i$ of $G$ corresponds to an indecomposable  projective module $P_i$ of the Brauer graph algebra.

First of all, let us recall equivalent conditions of tilting-discrete Brauer graph algebras.

\begin{theorem}\label{sdBGA}
%Let $\Lambda$ be the Brauer graph algebras of $G$.
The following are equivalent for a Brauer graph algebra $\Lambda$:
\begin{enumerate}
\item $\Lambda$ is tilting-discrete;
\item It is $\tau$-tilting finite;
\item The Cartan matrix of $\Lambda$ is positive definite;
\item It is regular;
\item The Grothendieck group $K_0(\smod\Lambda)$ of the stable module category of $\Lambda$ is finite;
\item The Brauer graph of $\Lambda$ has at most one cycle of odd length and none of even cycle.
\end{enumerate}
\end{theorem}

The equivalences of (3)--(6) were shown in \cite[Theorem 4.38]{JGM}.
The proof below on the implications $(5)\Leftrightarrow(6)$ is due to \cite{An2}; see also \cite[Remark 4.42]{JGM}.

\begin{proof}
The implications $(6)\Leftrightarrow(1)\Rightarrow(2)\Rightarrow(3)\Rightarrow(4)\Leftrightarrow(5)$ hold, respectively by \cite[Theorem 6.7]{AAC}, by definition, by \cite[Proposition 3.2]{H}, clear, and by \cite[Proposition 1]{TW}.% (see also \cite[Proposition 3.1]{G}).

We show the implications $(5)\Leftrightarrow(6)$ hold true.
Let $v$ and $n$ be the numbers of vertices and edges of $G$, respectively.
It follows from \cite[Theorem 6.1]{An2} that 
\[\begin{array}{c@{\iff}c}
K_0(\smod\Lambda)\ \mbox{has no free part} &
\begin{cases}
\ n=v-1 & \mbox{if $G$ is bipartite}; \\
\ n=v & \mbox{otherwise}.
\end{cases}
\end{array}\]
The former is just the case that $G$ is tree.
The latter is the case that $G$ has precisely one cycle and the cycle is of odd length.
This completes the proof.
\end{proof}

\begin{remark}
For a symmetric algebra of tubular type, the same equivalent conditions (1)--(5) of Theorem \ref{sdBGA} hold by \cite[Corollary 3.2]{AHMW}.
On the other hand, we know that this is not necessarily true for a symmetric algebra in general; neither the implications $(5)\Rightarrow(1)(2)$ nor $(3)\Rightarrow(1)(2)$ hold. %For example, consider group algebras (Example \ref{C3C3:C2}).
We will give concrete examples later.
\end{remark}

It is well-known that the trivial extension of an algebra $A$ is a Brauer graph algebra with multiplicity identically one if and only if $A$ is a gentle algebra \cite{Sc}.
Now, we determine gentle algebras whose trivial extensions are tilting-discrete.

\begin{theorem}\label{deriveddiscrete}
The following conditions are equivalent for a gentle algebra $A$:
\begin{enumerate}
\item The trivial extension of $A$ is a Brauer graph algebra whose Brauer graph has at most one cycle of odd length and none of even length;
\item It is tilting-discrete;
\item It is $\tau$-tilting finite;
\item $A$ has a positive definite Cartan matrix;
%\item The quadratic form $x^tCx$ is positive definite;
\item It is a derived-discrete algebra with nonsingular Cartan matrix.% but not an iterated tilted algebra of Dynkin type.
\end{enumerate}
\end{theorem}
\begin{proof}
%Note that the trivial extension $\Lambda$ of $A$ is a Brauer graph algebra with multiplicity identically one \cite{Sc},
%and its Brauer graph $G$ is not a tree because $A$ is nonsimply-connected \old{cite}.
%So, we see that $\Lambda$ is representation-infinite.
%
%In the case, by Theorem \ref{sdBGA}, $\Lambda$ is $\tau$-tilting finite if and only if $G$ has precisely one cycle and its length is odd.
%This is equivalent to the condition that $\Lambda$ is (representation-infinite) 1-domestic \cite{BS}; in particular, it is of polynomial growth.
%
%We can suppose that $A$ is not simply-connected.
%
%Assume that $A$ is $\tau$-tilting finite with positive definite Cartan matrix.
%satisfies the conditions (1)--(5) as in Proposition \ref{cyclotomic}.
%
Assume that $A$ has a positive definite Cartan matrix.
%Since $A$ is $\tau$-tilting finite, we obtain from \cite{Pl} that it is representation-finite.
By Theorem \ref{sdBGA} and \cite{BS}, we see that the trivial extension of $A$ is domestic. 
%If $A$ admits at least two cycles (in the Gabriel quiver), it follows from \cite{N} that $A$ is nondomestic, contradiction.
This implies from \cite{N} that the Gabriel quiver of $A$ admits at most one cycle.
%As $A$ is not simply-connected, we observe that the Gabriel quiver is not a tree, and so it has precisely one cycle; that is, $A$ is gentle one-cycle.
Let us suppose that $A$ is simply-connected.
If $A$ is representation-infinite, then it follows from \cite{ANS} that $A$ is an iterated tilted algebra of extended Dynkin type $\widetilde{\mathbb{D}}$ or $\widetilde{\mathbb{E}}$, a contradiction.
If $A$ is representation-finite, then we conclude that $A$ is an iterated tilted algebra of Dynkin type \cite{AS2}.

Suppose that $A$ is not simply-conncected; in particular, it is not a tree.
As above, we obtain that $A$ is gentle one-cycle.
If $A$ satisfies the clock condition, then it is piecewise hereditary of extended Dynkin type $\widetilde{\AA}$ by \cite[Theorem(A)]{AS1}, a contradition.
%We see that the Coxeter matrix of $A$ has eigenvalue one, which contradicts to Proposition \ref{cyclotomic}(5).
Thus, we derive that $A$ does not satisfy the clock condition, whence it is derived-discrete by \cite{V}.

Let $A$ be a derived-discrete algebra with nonsingular Cartan matrix but not a piecewise hereditary algebra of Dynkin type.
Thanks to \cite{BGS}, it is seen that $A$ is derived equivalent to the algebra presented by the quiver and relations for some $1\leq r\leq n$ and $m\geq0$:
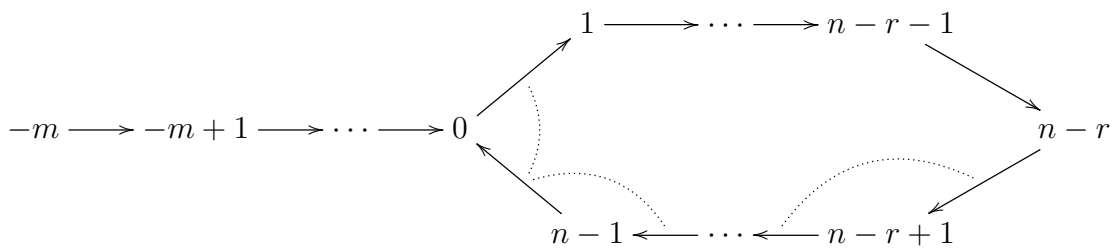
\[\xymatrix{
   &        &           &     & 1 \ar[r] &  \cdots \ar[r]  & n-r-1 \ar[rd] & \\
-m \ar[r] & -m+1 \ar[r] & \cdots \ar[r] & 0 \ar[ru]_{}="b"  &    &             &         & n-r \ar[ld]_{}="e"  \\
   &       &            &     & n-1 \ar[ul]_{}="a" &  \cdots \ar[l]_(0.4){}="c" & n-r+1 \ar[l]_(0.7){}="d" & 
\ar@/_/@{.}"a";"b"
\ar@/_1pc/@{.}"c";"a"
\ar@/_1.5pc/@{.}"e";"d"
}\]
Here, the dotted lines stand for the zero relations on two arrows.
In the case that $r<n$, it follows from \cite{BGS} that $r$ is odd, since $C$ is regular.
Assume that $r=n$.
We easily calculate the Cartan matrix of $A$ by hand, and then we have $\det C=2$ ($n$ is odd) and $\det C=0$ ($n$ is even).
Therefore, in any case, $r$ is odd.

Now, it is easy to see that the trivial extension $\Lambda$ of $A$ is derived equivalent to the Brauer graph algebra given by the Brauer graph:
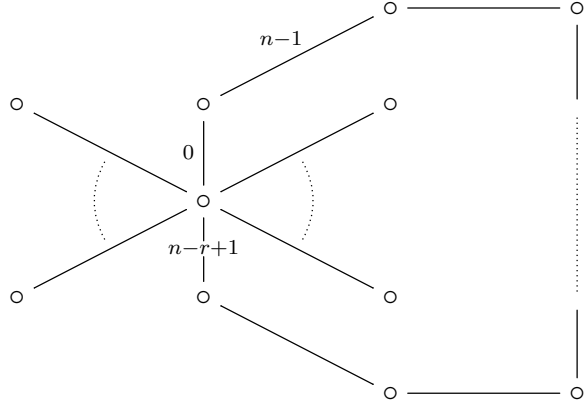
\[\xymatrix@C=2cm{
&& \circ \ar@{-}[dl]_{n-1}\ar@{-}[r] & \circ \ar@{-}[d] \\
\circ \ar@{-}[dr]_{}="a" & \circ \ar@{-}[d]_0 & \circ \ar@{-}[dl]^{}="b" & \ar@{.}[dd] \\
& \circ \ar@{-}[dl]_{}="c"\ar@{-}[d]|{n-r+1}\ar@{-}[dr]^{}="d" && \\
\circ & \circ \ar@{-}[dr] & \circ & \ar@{-}[d] \\
&& \circ \ar@{-}[r] & \circ
\ar@/_/@{.}"a";"c"
\ar@/^/@{.}"b";"d"
}\] 
Observe that the unique cycle is of length $r$, and as seen above $r$ is odd.
Thus, it turns out that $\Lambda$ is tilting-discrete by Theorem \ref{sdBGA}.
\end{proof}

%\begin{remark}\label{remark:bga}
%This proof says that the trivial extension of an algebra $A$ is a Brauer graph algebra whose Brauer graph has at most one cycle of odd length and none of even length if and only if $A$ is a derived-discrete algebra with regular Cartan matrix.
%Also, note that for a gentle algebra, the condition of its Cartan matrix being positive definite indicates its $\tau$-tilting finiteness; the converse does not necessarily hold.
%In fact, the radical-square-zero algebra of the quiver $\xymatrix{1 \ar@<2pt>[r] & 2 \ar@<2pt>[l]}$ is $\tau$-tilting finite, but it has a singular Cartan matrix.
%\end{remark}

Let $\Lambda$ be a (not necessarily, symmetric) algebra.
As is well-known, the set $\{P_1,\cdots, P_n\}$ of isomorphism classes of indecomposable projective modules forms a basis of the Grothendieck group $K_0(\Kb(\proj\Lambda))$ of $\Kb(\proj\Lambda)$;
so, we have $K_0(\Kb(\proj\Lambda))\simeq \Z^n$.
Under this basis, we denote by $[X]$ the (column) vector of a perfect complex $X$ in $K_0(\Kb(\proj\Lambda))$. %$K_0(\Kb(\proj\Lambda))$.

Let $T$ be a (basic) tilting complex with indecomposable decomposition $T_1\oplus\cdots\oplus T_n$.
The \emph{$g$-matrix} of $T$ is defined to be a matrix whose $j$-columns are $[T_j]$: say $\G_T$.

The following observation plays a crucial role.

\begin{lemma}\label{gmatrices}
Let $\Lambda$ be a (not necessarily, symmetric) algebra and $T$ its tilting complex.
Putting $\Gamma:=\End_{\Kb(\proj\Lambda)}(T)$, let $F:\Kb(\proj\Lambda)\to\Kb(\proj\Gamma)$ be a derived equivalence induced by $T$.
Then we have $\G_U=\G_T\cdot \G_{FU}$ for any tilting complex $U$ of $\Lambda$.
\end{lemma}

We can easily check the following observation.

\begin{lemma}\label{itm}
Let $\Lambda$ be the Brauer graph algebra of a Brauer graph $G$ and $i$ a nonleaf edge of $G$.
Then the left tilting mutation $T:=\mu_{P_i}^-(\Lambda)$ of $\Lambda$ with respect to $P_i$ is of the form:
\[T=\bigoplus\left\{
\begin{array}{c}
\xymatrix@R=2mm{
P_i \ar[r] & P_j\oplus P_k\ ;& \\
0 \ar[r] & P_\ell & (\ell\neq i).
}
\end{array}\right.\]
Here, the edges $j$ and $k$ are the predecessors of $i$ in G; possilby, $j=k$.
In particular, acting matrices to (row) vectors from right, the $g$-matrix $\G_T$ has an eigenvalue one with identically one eigenvector. 
\end{lemma}

For a silting complex $T$ of an algebra $\Lambda$, $\C_T$ denotes the connected component of the silting quiver of $\Lambda$ containing $T$.

Now, we give a new example of tilting-disconnected symmetric algebras.

\begin{theorem}\label{tdcBGA}
Assume that a Brauer graph $G$ is either of the following:
\begin{enumerate}[{\rm (i)}]
\item it has only one vertex and at least two edges, or;
\item it has precisely two vertices, at least two edges and is bipartite.
\end{enumerate}
Then the Brauer graph algebra $\Lambda$ of $G$ is not tilting-connected.
\end{theorem}

The idea of a proof comes from a private conversation with Grant, Iyama and the author circa 2010, in which we gave the first example of tilting-disconnected algebras; it was the Brauer graph algebra of digon.

\begin{proof}
Note first that $G$ owns no leaves.
By Lemma \ref{itm}, we see that the $g$-matrix of every irreducible tilting mutation of $\Lambda$ has eigenvalue one.
Since any Kauer move does not change the property of $G$ \cite{An},
it follows from Lemma \ref{gmatrices} that all the $g$-matrices of tilting complexes in $\C_\Lambda$ admit eigenvalue one.
However, of course, the 1-shift of $\Lambda$ is not really the case.
Thus, it turns out that the 1-shift of $\Lambda$ is not reachable to $\Lambda$ by iterated irreducible tilting mutation, whence $\Lambda$ is tilting-disconnected.
\end{proof}

We remark that this proof does not tell us how many connected components of the tilting quiver there exist; for now, we can not measure whether or not, even-shifts $\Lambda[2\Z]$ belong to the same component by only the information of eigenvalues of $g$-matrices.
Furthermore, it is still an open problem if such algebras are weakly silting-connected in the sense of August--Dugas \cite{AD}.

%\old{Remark for Antipov's Example}

%%%%%%%%%%%%%%%%%%%%%%%%%%%%%%%%%%%%%%%%%%%%%%%%%%%%%%%%%%%%%%%%%%%%%%%%%%%%%%%%%%%%%%%%%%%%%%%%%%%%%%%%%%%%%%%%
%\section*{Acknowledgements}

%%%%%%%%%%%%%%%%%%%%%%%%%%%%%%%%%%%%%%%%%%%%%%%%%%%%%%%%%%%%%%%%%%%%%%%%%%%%%%%%%%%%%%%%%%%%%%%%%%%%%%%%%%%%%%%%%%%%%%%%%%%%
\section{Trivial extension algebras}\label{sec:TEA}

In this section, we tackle the following problem.

\begin{question}\label{SDvsTauTF2}
Is a $\tau$-tilting finite symmetric algebra tilting-discrete?
\end{question}

%Of course, this is not necessarily true for a nonsymmetric algebra; the `if part' will fail.

This question does not necessarily have an affirmative answer.
Here is an example.

\begin{example}\label{Am}
Let $\ell, m\geq1$.
Let $\Lambda$ be the trivial extension of the algebra $A_m$ presented by the quiver $\xymatrix{1 \ar@(lu,ld)_x\ar@{=>}[r] & 2}$ with relations $x^m=0=xy$.
Here, the double arrow denotes $\ell$ arrows being drawn from 1 to 2 which are all labelled by $y$.
We remark that $A_1$ is the $\ell$-Kronecker algebra (see Proposition \ref{Kronecker}); so, we here exclude the case: suppose $m>1$. 

Then, $\Lambda$ is given by the quiver $\xymatrix{1 \ar@(ru, lu)_x\ar@(rd,ld)^a\ar@{=>}@<3pt>[r]^y & 2 \ar@{=>}@<3pt>[l]^z}$; we omit the relations.
%with relations $x^m=0=xy=a^2=ay=zx=za, xa=ax, y_iz_i=x^{m-1}a, z_jy_i=y_iz_j\ (i\neq j)$
Observe that $x$ and $a$ belong to the center of $\Lambda$, which implies that the $\tau$-tilting quivers of $\Lambda$ and $\Lambda/(x, a)$ coincide by \cite[Theorem 11]{EJR}.
In particular, $\Lambda$ is $\tau$-tilting finite if and only if $\ell=1$.
This is because $\Lambda/(x, a)$ is representation-finite; it is derived-discrete, and so silting-discrete.

Let $T$ be the (irreducible) left mutation of $A_m$ with respect to $P_2$; it is of the form 
\[\bigoplus{
\left\{
 \begin{array}{c}
  {\xymatrix@R=3mm{
   & P_1 \\
   P_2 \ar[r]^{(y)} & {P_1}^\ell
  }}
 \end{array}
\right.
}\]
Note that $T$ is a tilting complex (module), and so its endomorphism algebra $B_m$ is derived equivalent to $A_m$.
Moreover, $B_m$ is presented by the quiver $\xymatrix{1 \ar@{=>}@<3pt>[r]^z & 2 \ar@{=>}@<3pt>[l]^y}$ with relations
\[z_iy_j=0\ (i\neq j),\ z_iy_i=z_jy_j, (zy)^{m-1}z=0.\]
This leads to the fact that the trivial extension $\Gamma$ of $B_m$ has precisely $\ell$ added arrows to $B_m$ from 1 to 2.
In particular, $\Gamma$ admits a multiple arrow for any $\ell$, whence it is not $\tau$-tilting finite; so, it is not tilting-discrete, either. 
Since $\Lambda$ is derived equivalent to $\Gamma$, it turns out that $\Lambda$ is not tilting-discrete.
(Note that if $\ell=1=m$, then $\Lambda$ is tilting-discrete.)
\end{example}

We will focus on the algebra $A_m$ in the subsection \ref{subsec:Am}.

Thus, a next problem is when a $\tau$-tilting finite symmetric algebra is tilting-discrete.

As is seen in the introduction, if the trivial extension of an algebra $A$ is $\tau$-tilting finite, then (1) $A$ is $\tau$-tilting finite and (2) has a positive definite Cartan matrix.
%Let us recall necessary conditions for trivial extension algebras to be $\tau$-tilting finite.
%
%\begin{proposition}\label{TauTFandPDC}
%Let $A$ be an algebra and $C$ its Cartan matrix.
%If the trivial extension of $A$ is $\tau$-tilting finite, then we have the following:
%\begin{enumerate}
%\item $A$ is $\tau$-tilting finite;
%\item The quadratic form $x^\tr Cx$ is positive definite.
%\end{enumerate}
%\end{proposition}
%
%We refer to the condition (2) of this proposition as a \emph{positive definite Cartan matrix}.
In this case, the Coxeter matrix of $A$ satisfies the following conditions (see Appendix \ref{append:cyclotomic}):
\begin{enumerate}[(i)]
\item it has no eigenvalue one, but all eigenvalues have absolute value one;
\item it is diagonalizable.
%\item It has no eigenvalue one.
\end{enumerate}

%We check that the conditions (1) and (2) as in Proposition \ref{TauTFandPDC} are independent.
We check the conditions (1) and (2) above are independent.

\begin{example}\label{TauTFvsPDC}
The implication (1)$\Rightarrow$(2) is not always true.
In fact, the radical-square-zero algebra given by the quiver $\xymatrix{1 \ar@<2pt>[r] & 2 \ar@<2pt>[l]}$, which is gentle, is $\tau$-tilting finite (more strongly, silting-discrete),  but it has singular Cartan matrix.

The condition (2) does not necessarily indicate the condition (1), either.
Indeed, let $A$ be the algebra given by the quiver $\xymatrix{1 \ar@<4pt>[r]^x\ar[r]|y & 2 \ar@<4pt>[l]^z}$ with relations $xz=0=yz$ \cite[Example 2.1(2)]{S}.
Then, this is $\tau$-tilting infinite,
%; the quiver admits a multiple arrow \cite[Theorem 5.12(d)]{DIRRT}.
but the Cartan matrix $\begin{pmatrix} 1 & 1 \\ 2 & 3\end{pmatrix}$ satisfies the condition (2).
%and the Coxeter matrix $\begin{pmatrix} 1 & -1 \\ 3 & -2\end{pmatrix}$.
%\old{Further, we infer that every algebra with the same trivial extension as $A$ fulfills (2).}
\end{example}

We might naturally hope that the trivial extension of a $\tau$-tilting finite algebra with positive definite Cartan matrix
%satisfying the conditions (1)--(5) as in Proposition \ref{cyclotomic}
is at least $\tau$-tilting finite.
However, this fails in general; the following algebra additionally has finite global dimension; so, it is of cyclotomic type.
%In fact, the algebra $A_m$ for $\ell=1$ defined in Example \ref{Am} is $\tau$-tilting finite (more strongly, representation-finite) and has a positive definite Cartan matrix, but its trivial extension is not tilting-discrete; $A_m$ is of generalized cyclotomic type.
%
%It also fails even if $A$ has finite global dimension; in the case, $A$ is of cyclotomic type.
%For instance, let us observe the following algebra.

\begin{example}\label{TauTFwPDbutSID}
Let $A$ be the radical-square-zero algebra given by the quiver:
\[\xymatrix{
1 \ar[r]\ar[d]\ar[dr] & 2 \ar[d] \\
3 \ar[r] & 4
}\]
Note that this is representation-infinite.
Its Cartan and Coxeter matrices are the following:
\[\begin{array}{c@{\hspace{5mm}\mbox{and}\hspace{5mm}}c@{\ \xrightarrow{\mbox{{\scriptsize Cox. Poly.}}}\ }c}
{C=\begin{pmatrix}
1 & 0 & 0 & 0 \\
1 & 1 & 0 & 0 \\
1 & 0 & 1 & 0 \\
1 & 1 & 1 & 1
\end{pmatrix}} &
{\Phi=\begin{pmatrix}
0 & 0 & 0 & -1 \\
0 & 0 & 1 & -1 \\
0 & 1 & 0 & -1 \\
-1 & 1 & 1 & -1 \\
\end{pmatrix}} &
(x+1)^2(x^2-x+1)
\end{array}\]
Thanks to \cite[Theorem 3.1]{Ad1}, we observe that $A$ is $\tau$-tilting finite.
Moreover, we can easily check that the form $x^\tr Cx$ is positive definite;
hence, $A$ can be chosen as our algebra.
%satisfies all the conditions as in Proposition \ref{cyclotomic}.

On the other hand, the trivial extension $\Lambda$ of $A$ admits a path algebra of extended Dynkin type $\widetilde{\AA}$ as a factor, which is $\tau$-tilting infinite, and hence, so is $\Lambda$.

We will show that $A$ is silting-indiscrete in the subsection \ref{ENSA1}.
\end{example}

Even if any algebra with isomorphic trivial extension is $\tau$-tilting finite and has a positive definite Cartan matrix, the trivial extension is not necessarily $\tau$-tilting finite.

\begin{example}\label{withsameTE}
Let $A$ be the algebra given by the quiver $\xymatrix{1 \ar@<2pt>[r]^z & 2 \ar@<2pt>[l]^y}$ with relation $zyz=0$; this is $B_2$ for $\ell=1$ as in Example \ref{Am}.
Observe that $A$ is $\tau$-tilting finite and has a positive definite Cartan matrix.
If another algebra $B$ has isomorphic trivial extension as $A$, then there are an algebra $R$ and an $R$-$R$ bimodule $M$ such that $A\simeq R\ltimes M$ and $B\simeq R\ltimes DM$ \cite{Wak}.
In case, $R$ must be the path algebra $\xymatrix{1 & 2 \ar[l]^y}$, and then we have $M=DP_2\otimes P_2$; so, $B$ is isomorphic to $A$.
However, the trivial extension of $A$ is not $\tau$-tilting finite.
\end{example}

Now, we give several cases in which Question \ref{SDvsTauTF2} has an affirmative answer.

As we have already seen in Theorem \ref{sdBGA} (\ref{deriveddiscrete}), the first example is the case of Brauer graph algebras (the trivial extensions of gentle algebras).

%\begin{corollary}
%Any Brauer graph algebra is tilting-discrete iff it is $\tau$-tilting finite.
%\end{corollary}

The trivial extension of a piecewise hereditary algebra of Dynkin type (abbr. PHAD) is symmetric and representation-finite, and so it is tilting-discrete \cite[Theorem 5.2]{Ai}.
Thus, we next consider when a $\tau$-tilting finite algebra with positive definite Cartan matrix is a PHAD.
%\old{One says that a class of algebras \emph{has psudo-hereditary} property (of Dynkin type) if in the class, a $\tau$-tilting finite algebra with positive definite Cartan matrix is a PHAD; as a terminology, an algebra in such a class is also called \emph{pseudo-hereditary}.}
Let $\H$ denote the (Morita equivalence) class of algebras which are either $\tau$-tilting infintie, having non-positive definite Cartan matrices, or PHADs.
%\old{`pseudo-hereditary'というワードはキモい}
%
We put the following theorem as a slogan.

\begin{theorem}\label{pseudohereditary}
The following are equivalent for $A\in\H$:
\begin{enumerate}
\item The trivial extension of $A$ is tilting-discrete;
\item It is $\tau$-tilting finite;
\item $A$ is a $\tau$-tilting finite algebra with positive definite Cartan matrix;
\item It is a piecewise hereditary algebra of Dynkin type.
\end{enumerate}
\end{theorem}

The class $\H$ contains good classes of algebras and is closed under certain properties.
We will explore which algebras belong to $\H$ in Appendix \ref{append:phad}.
For example, we obtain the following result by a similar argument to Proposition \ref{stableequivalent}.

\begin{corollary}
Let $\Lambda$ be a symmetric algebra stable equivalent to the trivial extension of a derived hereditary algebra.
Then $\Lambda$ is tilting-discrete if and only if it is $\tau$-tilting finite.
\end{corollary}

Let $Q$ be an acyclic quiver.
A \emph{cluster-tilted algebra of type} $Q$ is the endomorphism algebra of a cluster-tilting object of the cluster category $\C_Q:=\Db(\mod KQ)/F$, where $F:=\tau^{-1}\circ[1]$.
As is well-known, a cluster-tilted algebra is given by the relation extension $R\ltimes\Ext_R^2(DR, R)$ of some tilted algebra $R$ of type $Q$ and is always at most 1-Iwanaga--Gorenstein and has global dimension 1 or $\infty$.
Moreover, we also know that a cluster-tilted algebra of type $Q$ is representation-finite if and only if $Q$ is of Dynkin type.
In the case, its Cartan matrix admits only $\{0, 1\}$-entries.
We refer to \cite{As, ABS, BMR} for more details.

\begin{theorem}\label{clustertilted}
Let $A$ be a nonhereditary cluster-tilted algebra of type $Q$. Then the following conditions are equivalent:
\begin{enumerate}
\item The trivial extension of $A$ is tilting-discrete;
\item It is $\tau$-tilting finite;
\item $Q$ is of Dynkin type $\AA$ and $A$ has precisely one oriented 3-cycle.
\end{enumerate}
\end{theorem}
\begin{proof}
As above, put $A:=R\ltimes\Ext_R^2(DR, R)$ for a tilted algebra $R$ of type $Q$.
Assume that the trivial extension of $A$ is $\tau$-tilting finite.
Then, $A$, and hence $R$, are $\tau$-tilting finite as well, which implies that $Q$ is of Dynkin type; so, $A$ is representatin-finite \cite{Z}.

Let us check that $Q$ is neither of type $\D$ nor $\EE$.
To discuss inductively, we consider the quotient $A/AeA$ of $A$ by a primitive idempotent $e$, which is also cluster-tilted \cite{BMR2} and $\tau$-tilting finite.
Note that $A$ has a positive definite Cartan matrix, but $A/AeA$ might not.
Let $P$ and $S$ be the indecomposable projective and simple modules corresponding to $e$, respectively.
If $S$ appears at most only in the socle of $A/P$ (e.g., $e$ is a sink), then the Cartan matrix of $A/AeA$ is positive definite.
Thus, we may suppose that $A$ does not have a sink, a source, or an oriented 3-cycle at an end; i.e., $Q$ is not of the form
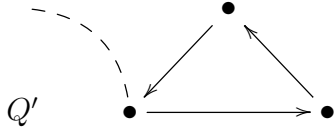
\[\xymatrix{
 \ar@/^1pc/@{--}[dr]     &        & \bullet \ar[dl] & \\
Q'  & \bullet \ar[rr] &            & \bullet \ar[ul]   
}\]

Applying these inductive steps, we can easily calculate their Cartan matrices;
thanks to the derived equivalence classifications for the cluster-tilted algebras of Dynkin type $\D$ and $\EE$ \cite{BHL, BHL2}.
Then, it turns out that they have non-positive definite Cartan matrices.

Finally, we remark that the cluster-tilted algebras of Dynkin type $\AA$ are gentle \cite{ABCJP}; see \cite{BV} for their derived equivalence classification.
So, we are done by Theorem \ref{deriveddiscrete}.
\end{proof}

\if0
\old{
trivial extensionが$\tau$-tilting finite or tilting-discreteとなる例がもっとほしい.

\begin{question}
Let $A$ be a selfinjective algebra.
Then the trivial extension of $A$ is $\tau$-tilting finite (resp. silting-discrete) if and only if so is $A$ and it has a positive definite Cartan matrix.
\end{question}

\begin{question}
Let $A$ be an FCY algebra of finite global dimension.
Then the following are equivalent???
\begin{enumerate}
\item The trivial extension of $A$ is tilting-discrete;
\item It is $\tau$-tilting finite;
\item $A$ is what...
\end{enumerate}
\end{question}
}
\fi

\if0
\old{蛇足
The algebra as in Example \ref{TauTFwPDbutSID} was not the desired one (even though it is $\tau$-tilting finite and has a positive definite Cartan matrix), but let us construct an algebra we desire with the same Cartan matrix.

\begin{example}
Let $A$ be the algebra presented by the quiver with relation $xy=0$:
\[\xymatrix{
1 \ar[r]^x\ar[d] & 2 \ar[d]^y \\
3 \ar[r] & 4
}\]
Observe that this algebra has the same Cartan matrix as the algebra of Example \ref{TauTFwPDbutSID}.
Moreover, the algebra is gentle one-cycle not satisfying the clock condition, and so derived-discrete; indeed, it is derived equivalent to the algebra given by $\xymatrix{2 \ar[r] & 4 \ar[r] & 3 \ar@<2pt>[r]^b & 1 \ar@<2pt>[l]^a}$ with $ab=0$.
The trivial extension of $A$ is the Brauer graph algebra presented by
\[\xymatrix@R=5mm @C=1.5cm{
& & \circ \ar@{-}[dd]|2 \\
\circ \ar@{-}[r]|3 & \circ \ar@{-}[ur]|1\ar@{-}[dr]|4 & \\
& & \circ
}\]
Thus, $A$ admits the tilting-discrete trivial extension algebra.
\end{example}
}
\fi

All the cases we have seen above deduce that $A$ is representation-finite and silting-discrete when the trivial extension of $A$ is tilting-discrete.
We pose the following question.

\begin{question}
If the trivial extension of $A$ is tilting-discrete, then is $A$ silting-discrete?
Conversely, is the trivial extension of a silting-discrete algebra with positive definite Cartan matrix tilting-discrete (at least, $\tau$-tilting finite)?
\end{question}

We remark that the $\tau$-tilting finiteness of the trivial extension of an algebra $A$ does not necessarily yields the silting-discreteness of $A$; see Example \ref{Am} and Proposition \ref{AmSID}.

\section{An upper bound of Cartan invariants}

In this section, we study a relationship beween the tilting-discreteness of symmetric algebras and upper bounds of Cartan invariants, which was first given in \cite{EJR}.

Let $\Lambda$ be a symmetric algebra and $C$ its Cartan matrix.
Denote by $\II$ the set of the isomorphism classes of indecomposable partial tilting complexes of $\Lambda$ and take the map
$\delta: \II\to K_0(\Kb(\proj\Lambda))\to\Z\ (X\mapsto [X]^\tr\cdot C\cdot [X])$; note that $\delta(X)=\dim\End_{\Kb(\proj\Lambda)}(X)$.
We remark that the supremum of $\delta(\II)$ is a derived invariant; the Cartan matrices of derived equivalent algebras are $\Z$-congruent by the $g$-matrix of a tilting complex.

For an integer $z$, we set $\RR_z:=\{v\in\Z^n\ |\ v^\tr\cdot C\cdot v=z \}$.

We recall the following result, which is a slight generalization of \cite[Theorem 13]{EJR}.

\begin{theorem}\label{upperbound}
Assume that $\Lambda$ is symmetric and has a positive definite Cartan matrix.
If there is an upper bounds of $\delta(\II)$, then $\Lambda$ is tilting-discrete.
In particular, if there exist only finitely many derived equivalent algebras to $\Lambda$ (up to Morita equivalence), then $\Lambda$ has a positive definite Cartan matrix if and only if $\Lambda$ is tilting-discrete.
\end{theorem}
\begin{proof}
By assumption, note that $\RR_z$ is a finite set for any $z$; it is empty for every $z\leq0$. 
Put $d:=\sup\delta(\II)<\infty$.
%We also remark that for any indecomposable partial tilting complex $X$, $\delta(X)$ coincides with the dimension of the endomorphism algebra of $X$.
%$\End_{\Kb(\proj\Lambda)}(X)$.
%
Let $T$ be a tilting complex of $\Lambda$ and $G: \Kb(\proj\Gamma)\to\Kb(\proj\Lambda)$ an induced derived equivalence, where $\Gamma:=\End_{\Kb(\proj\Lambda)}(T)$.

Suppose that $T=\Lambda$.
As above, the set of vectors $[X]$ for indecomposable 2-term pretilting complexes $X$ is finite.
Thanks to \cite[Theorem 6.5]{DIJ}, it is seen that there exist only finitely many indecomposable 2-term tilting complexes, whence $\Lambda$ is $\tau$-tilting finite.

Let us consider a general case.
Take an indecomposable 2-term pretilting complex $U$ of $\Gamma$.
Then, we get in $K_0(\Kb(\proj\Lambda))$, $[GU]=\G_T\cdot [U]$, where $[U]$ is the $g$-vecter of $U$ in $K_0(\Kb(\proj\Gamma))$.
By assumption, one obtains that $[GU]$ belongs to some $\RR_z$ for $1\leq z\leq d$, and so when $U$ runs all such complexes, the vectors appear only finitely many numbers.
If $[GU]=[GV]$ for another such a complex $V$ of $\Gamma$, then we have $\G_T\cdot [U]=\G_T\cdot [V]$; hence, $[U]=[V]$ in $K_0(\Kb(\proj\Gamma))$ because $\G_T$ is invertible.
So, we get an isomorphism $U\simeq V$ by \cite[Theorem 6.5]{DIJ}, which implies that $\Gamma$ is $\tau$-tilting finite.
Thus, $\Lambda$ is tilting-discrete.
%
%The last assertion follows from \cite[Proposition 3.2]{H} (the `if' part) and \cite[Theorem 13]{EJR} (the `only if' part).
\end{proof}

\begin{remark}
The assumption of the upper bound as in Theorem \ref{upperbound} can be restricted, instead of $\II$, to the set of indecomposable direct summands of iterated (irreducible left) mutations starting from $\Lambda$; see \cite[Theorem 2.4]{AM}. 
\end{remark}

Thanks to \cite[Proposition 3.2]{H}, we immediately obtain the following corollary.

\begin{corollary}\label{upperbound2}
A $\tau$-tilting finite symmetric algebra with $\sup\delta(\II)<\infty$ is tilting-discrete.
\end{corollary}

%\begin{remark}
%Theorem \ref{upperbound} says that a point to check tilting-discreteness is to observe the diagonal entries of Cartan matrices of all derived equivalent algebras to a given symmetric algebra.
We know that, under the assumption of the derived equivalence class of a symmetric algebra being finite, the conditions (1)--(3) as in Theorem \ref{sdBGA} are equivalent.
%; actually, Brauer graph algebras are the case.
% admits only finitely many derived equivalent algebras.
%
However, we have no idea if a tilting-discrete symmetric algebra satisfies the assumption.
More weakly, it is also unknown whether or not, the tilting-discreteness brings about an upper bound of $\delta(\II)$; that is, does the converse of Corollary \ref{upperbound2} hold?
%\end{remark}

We answer to a question in \cite{R}, which asks the derived equivalence class containing the following algebra $\Lambda$.

\begin{example}
Let $\Lambda$ be the algebra given by the quiver $\xymatrix{1 \ar@<2pt>[r]^a\ar@(lu, ld)_x & 2 \ar@<2pt>[l]^a\ar@(ru, rd)^{x}}$ with relations $x^3=0=a^3$ and $xa=ax$.
If $K$ has characteristic 3, then this algebra is isomorphic to the group algebra of $(C_3\times C_3)\rtimes C_2$; the action of $C_2$ on $C_3\times C_3$ switches the generator of $C_3$ to each other, and then the group is isomorphic to $S_3\times C_3$.

The Cartan matrix of $\Lambda$ is $\begin{pmatrix} 6 & 3 \\ 3 & 6 \end{pmatrix}$, positive definite.
We see that $\Lambda$ is isomorphic to $\overline{\Lambda}\otimes_KK[x]/(x^3)$, 
where $\overline{\Lambda}$ is presented by the quiver $\xymatrix{1 \ar@<2pt>[r]^a & 2 \ar@<2pt>[l]^a}$ with $a^3=0$.
One obtains from \cite[Theorem 2.1]{AH} that the functor $F:=-\otimes_KK[x]/(x^3)$ gives rise to isomorphisms $\tilt\overline{\Lambda}\xrightarrow{\sim}\tilt\Lambda$ and $\End_{\Kb(\proj\Lambda)}(FT)\simeq F\End_{\Kb(\proj\overline{\Lambda})}(T)$ of posets and algebras, respectively.
%\[\begin{array}{c@{\hspace{1cm}\mbox{and}\hspace{1cm}}c}
%\tilt\overline{\Lambda}\xrightarrow{\sim}\tilt\Lambda &
%\End_{\Kb(\proj\Lambda)}(FT)\simeq F\End_{\Kb(\proj\overline{\Lambda})}(T).
%\end{array}\]
Since $\overline{\Lambda}$ is the Brauer tree algebra of line with 2 edges, it turns out that the derived equivalence class consists of only itself; hence, so is $\Lambda$.

Note that if the characteristic of $K$ is 3, then $\Lambda$ is actually isomorphic to $KS_3\otimes KC_3$, but our discussion is independent of the characteristic.
\end{example}

Let us observe an example of $\tau$-tilting infinite symmetric algebras with positive definite Cartan matrix \cite{R}; Theorem \ref{upperbound} leads to the fact that such an algebra admits infinitely many derived equivalent algebras (up to Morita equivalence). 

\begin{example}\label{C3C3:C2}
Let $\Lambda$ be the algebra presented by the quiver $\xymatrix{1 \ar@<5pt>[r]\ar@<2pt>[r] & 2 \ar@<2pt>[l]\ar@<5pt>[l]}$ with certain relations; we omit but illustrate them by the Loewy structure, and also put its Cartan matrix and the quadratic form:
\[\begin{array}{c@{\hspace{2cm}}c@{\hspace{2cm}}c}
\vcenter{\xymatrix@C=5mm @R=5mm{
&& \circ \ar@{-}[dl]\ar@{-}[dr] && \\
& \bullet \ar@{-}[dl]\ar@{-}[dr] && \bullet \ar@{-}[dl]\ar@{-}[dr] & \\
\circ \ar@{-}[dr] && \circ \ar@{-}[dl]\ar@{-}[dr] && \circ \ar@{-}[dl] \\
& \bullet \ar@{-}[dr] && \bullet \ar@{-}[dl] & \\
&&\circ && 
}} &
{\begin{pmatrix} 5 & 4 \\ 4 & 5 \end{pmatrix}} &
5x^2+8xy+5y^2
\end{array}
\]
If the characteristic of $K$ is 3, then this is obtained by the group algebra of $(C_3\times C_3)\rtimes C_2$; the action of $C_2$ on $C_3\times C_3$ sends each element to its inverse.

Observe that $\Lambda$ is $\tau$-tilting infinite and has a positive definite Cartan matrix.
For each integer $a\geq0$, we have an indecomposable 2-term pretilting complex $X_a:=[P_1^a\to P_2^{a+1}]$; this appears as direct summands of iterated irreducible left mutations starting from $\Lambda$ with respect to the first and the second summands, alternately.
Thus, we get $\delta(X_a)=2a^2+2a+5$; so, there is no upper bound of $\delta(\II)$.
Furthermore, $\Lambda$ admits infinitely many derived equivalent algebras.
\end{example}

Note that, while Rickard showed this fact in \cite{R} by calculating iterated (irreducible) tilting mutations of the algebra with respect to the same vertex, our constructions are given by iterated (irreducible) tilting mutations of the algebra at the two vertices alterately.

The trivial extension of the algebra $A$ as in Example \ref{TauTFwPDbutSID} is also the case similar to Example \ref{C3C3:C2}; we leave it to the reader as an exercise (\cite[Theorem 2.3]{R} is useful).

Thanks to \cite{E}, we give a reduction technique for the silting quiver of a certain algebra.

\begin{proposition}\label{reductiontoBGAofdigon}
Fix $\ell\geq2$.
Let $m\geq1$ and $\Lambda_m$ be the algebra presented by the following quiver with relations:
\[\begin{array}{c@{\hspace{5mm}:\hspace{5mm}}c}
\vcenter{\xymatrix{1 \ar@<7pt>[r]^x\ar@<3pt>[r]|y & 2 \ar@<3pt>[l]|x\ar@<7pt>[l]^y}} &
x^{2m}=0=y^\ell\ \mbox{and}\ yx=xy.
\end{array}
\]
Then we have an isomorphism $\silt\Lambda_m\simeq \silt\Lambda_1$ of posets.
\end{proposition}
\begin{proof}
For a (finite) quiver $Q$, we denote by $\widehat{KQ}$ the completion of $KQ$ with respect to the ideal generated by the arrow set $Q_1$, and the corresponding completions of ideals are denoted by a horizontal bar.

Now, let $Q$ be our quiver and $I^\flat$ denote the ideal of $KQ$ generated by $y^\ell$ and $yx-xy$.
Put $z:=x^{2m}+x^{2m}$, where the former $x^{2m}$ is the path from 1 to 1 and the latter $x^{2m}$ is the path from 2 to 2.
Note that $z$ is in the center of the (infinite dimensional) $K$-algebra $KQ/I^\flat$ and the $K[z]$-module $KQ/I^\flat$ is finitely generated and free; considering the paths (starting from each vertex)
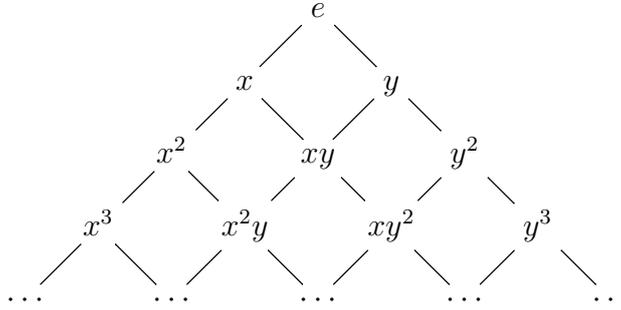
\[\xymatrix%@C=5mm@R=5mm
@!@R=1mm@C=1mm{
& &&& e \ar@{-}[dl]\ar@{-}[dr] &&& & \\
&&  & x \ar@{-}[dl]\ar@{-}[dr] && y \ar@{-}[dl]\ar@{-}[dr] &  & &\\
&& x^2 \ar@{-}[dl]\ar@{-}[dr] && xy \ar@{-}[dl]\ar@{-}[dr] && y^2 \ar@{-}[dl]\ar@{-}[dr] & & \\
& x^3 \ar@{-}[dl]\ar@{-}[dr] & & x^2y \ar@{-}[dl]\ar@{-}[dr] && xy^2 \ar@{-}[dl]\ar@{-}[dr] && y^3\ar@{-}[dl]\ar@{-}[dr] & \\
\cdots && \cdots && \cdots && \cdots && \cdots
}\]
the top right of the line $x^{2m}\cdot-$ forms a generating set.
This shows that $KQ/I^\flat$ is a $K[z]$-order.
Thus, by letting $X$ act as $z$, the (infinite dimensional) K-algebra $\Gamma:=\widehat{KQ}/\overline{I^\flat}$ becomes a $K\dbracket{X}$-order.
Moreover, we have an isomorphism $\Lambda_m\simeq \Gamma/X\Gamma$ of algebras.
Consequently, the assertion follows from \cite[Corollary 6.4]{E}.
\end{proof}

We remark that $\Lambda_m$ is a selfinjective algebra with singular Cartan matrix; in particular, it is symmetric iff $\ell$ is even.
When $\ell=2$, we observe that $\Lambda_1$ is isomorphic to the Brauer graph algebra of digon.
So, the following corollary is immediately obtained.

\begin{corollary}
Keep the same setting as in Proposition \ref{reductiontoBGAofdigon} and assume that $\ell$ is even.
Then $\Lambda_m$ is not tilting-connected. 
\end{corollary}

Unfortunately, we do not know a way to reduce the algebra as in Example \ref{C3C3:C2} so that the reduction has the same silting quiver; so, one asks if the algebra is tilting-connected.
We only comprehend that the algebra admits the same $\tau$-tilting quiver as the Brauer graph algebra of digon by \cite[Theorem 11]{EJR}.

\if0
\old{
Although the algebra of this example is not tilting-discrete, we can easily obtain a tilting-discrete symmetric algebra with the same Cartan matrix; e.g., the Brauer tree algebra with two edges and with exceptional vertex of multiplicity 4 in the center.
}
\fi

%\begin{remark}
%Of course, 
Evidently, we can not drop the assumption of the positive definiteness as in Theorem \ref{upperbound}; in fact, the Brauer graph algebra of digon admits the derived equivalence class consisting of only itself and is not tilting-discrete; that is, $\sup\delta(\II)=2$ but there are infinitely many indecomposable (2-term) pretilting complexes $X$ with $\delta(X)=2$.%$|\RR_2|=\infty$.
%\end{remark}

Moreover, even if the derived equivalence class of a given symmetric algebra is finite up to Morita equivalence, all the conditions as in Theorem \ref{sdBGA} are not necessarily equivalent;
the regularity of the Catan matrix does not always indicate the $\tau$-tilting finiteness.

\begin{theorem}\label{Kronecker}
Let $\Lambda$ be the trivial extension of the $\ell$-Kronecker algebra for $\ell\geq1$.
Then the Morita and derived equivalence classes of $\Lambda$ coincide.
%any algebra derived equivalent to $\Lambda$ is Morita equivalent to it.
Furthermore, we have:
\begin{enumerate}
\item If $\ell=1$, then $\Lambda$ is tilting-discrete;
\item If $\ell=2$, then $\Lambda$ is not tilting-connected and has a singular Cartan matrix;
\item Otherwise, $\Lambda$ is not tilting-connected and has a regular (indefinite) Cartan matrix.
\end{enumerate}
\end{theorem}
\begin{proof}
The first assertion follows from \cite{PX}.
It remains only to show the case (3).
Labelling the vertices as 1 and 2, we observe that the Cartan matrix of $\Lambda$ is $\begin{pmatrix} 2 & \ell \\ \ell & 2 \end{pmatrix}$, regular and indefinite.
Let $T$ and $U$ be the left mutations of $\Lambda$ with respect to $P_1$ and $P_2$, respectively.
Then, we calculate their $g$-matrices $\G_T=\begin{pmatrix} -1 & 0 \\ \ell & 1\end{pmatrix}$ and $\G_U=\begin{pmatrix} 1 & \ell \\ 0 & -1\end{pmatrix}$.
Since the derived equivalence class consists of only $\Lambda$ (and the vertices are not distinguish), by Lemma \ref{gmatrices} we have a surjective map from $\C_\Lambda$ to the group generated by $\G_T$ and $\G_U$ which sends $X$ to $\G_X$.
It is not difficult to see that neither $-E$ nor $\begin{pmatrix} 0 & -1 \\ -1 & 0\end{pmatrix}$ is a member of the group.
Therefore, we find out that the 1-shift of $\Lambda$ does not belong to $\C_\Lambda$; so done.
%whence $\Lambda$ is tilting-disconnected.
\end{proof}
\if0
\begin{example}\label{regularTID}% We got a more general result (above) of this example.
Let $\Lambda$ be the trivial extension algebra of the path algebra given by the quiver $\xymatrix{1 \ar@<4pt>[r]\ar@<0pt>[r]\ar@<-4pt>[r] & 2}$.
Observe that the Cartan matrix of $\Lambda$ is $\begin{pmatrix} 2 & 3 \\ 3 & 2\end{pmatrix}$, regular and indefinite.
It follows from \cite{PX} that the derived equivalence class consists of only $\Lambda$ itself; this implies that $\sup\delta(\II)=2$.
However, $\Lambda$ is neither $\tau$-tilting finite nor tilting-discrete.

We mention that $\Lambda$ is not tilting-connected, either.
Let $T$ and $U$ be the irreducible left mutations of $\Lambda$ with respect to $P_1$ and $P_2$, respectively.
Then, we calculate their $g$-matrices $\G_T=\begin{pmatrix} -1 & 0 \\ 3 & 1\end{pmatrix}$ and $\G_U=\begin{pmatrix} 1 & 3 \\ 0 & -1\end{pmatrix}$.
Since the derived equivalence class consists of only $\Lambda$ (and the vertices are not distinguish), by Lemma \ref{gmatrices} we have a surjective map from $\C_\Lambda$ to the group generated by $\G_T$ and $\G_U$ which sends $X$ to $\G_X$.
It is not difficult to see that neither $-E$ nor $\begin{pmatrix} 0 & -1 \\ -1 & 0\end{pmatrix}$ is not a member of the group.
Consequently, we find out that the 1-shift of $\Lambda$ does not belong to $\C_\Lambda$, whence $\Lambda$ is tilting-disconnected.
\end{example}
\fi

Let us revisit Example \ref{Am} to observe the supremum of $\delta(\II)$; note that the class of the algebras $A_m$ contains the generalized Kronecker algebras ($m=1$).

\begin{example}\label{singularandnoupperbound}
%Let $\ell, m\geq1$.
%Let $\Lambda$ be the trivial extension of the algebra $A_m$ presented by the quiver $\xymatrix{1 \ar@(lu,ld)_x\ar@{=>}[r] & 2}$ with relations $x^m=0=xy$.
%Here, the double arrow denotes $\ell$ arrows being drawn from 1 to 2 which are all labelled by $y$.
%Note that the case that $m=1$ is just Proposition \ref{Kronecker}.
Let $A_m$ be the algebra defined in Example \ref{Am} and $\Lambda$ its trivial extension.
Then the Cartan matrix of $\Lambda$ is $\begin{pmatrix} 2m & \ell \\ \ell & 2\end{pmatrix}$; this can be positive definite, indefinite or positive semidefinite (singular).
%for example, it is singular if $m=4=\ell$ and is indefinite if $m=2$ and $\ell=3$.

As $\Lambda$ owns the same $\tau$-tilting quiver with the trivial extension of the $\ell$-Kronecker algebra by \cite[Theorem 11]{EJR}, we can obtain indecomposable 2-term pretilting complexes $X_t$ with $g$-vector $\begin{pmatrix} a_{t+1} \\ -a_t \end{pmatrix}$, where
%$\alpha_{\pm}=2\pm\sqrt{3}$
$\alpha_{\pm}=\frac{\ell\pm\sqrt{\ell^2-4}}{2}$
and $a_t=\frac{\alpha_+^{t-1}-\alpha_-^{t-1}}{\alpha_+-\alpha_-}$ if $\ell>2$ and $a_t=t-1$ if $\ell=2$.
So, we have
\[\delta(X_t) =
\begin{cases}
\ 2\cdot\frac{(m-1)(\alpha_+^{2t}+\alpha_-^{2t}-2)+(\ell^2-4)}{\ell^2-4} & (\ell>2) \\
\ 2\cdot \{(m-1)\cdot t^2+1\} & (\ell=2)
\end{cases}\ \xrightarrow{t\to\infty}\ \infty\ \mbox{(if $m\neq1$)}.
\]
Consequently, one gets $\sup\delta(\II)=\infty$ if $m, \ell\geq2$.
Note that this fact is proved by using alternate (irreducible) tilting mutations.

Let $m\geq2$ and $\ell=1$. As is seen in Example \ref{Am}, the values of $\delta(X)$ do not diverge to infinity by altenate (irreducible) tilting mutations from $\Lambda$.
On the other hand, we observe by \cite[Theorem 2.3]{R} that they grow infinitely by using iterated irreducible left mutations at the second summand.
Indeed, the $t$th-times irreducible left mutation has the form
\[\bigoplus\left\{
\begin{array}{c}
 \xymatrix@R=3mm{
 &&&& P_1 \\
 P_2 \ar[r]_y & P_1 \ar[r]_{\left(\begin{smallmatrix} x \\ a \end{smallmatrix}\right)} & {P_1}^2 \ar[r] & \cdots \ar[r] & {P_1}^t
 }
\end{array}
\right.\]
We also obtain that iterated irreducible left mutations with respect to the first summand admits $\Lambda$ as the endomorphism algebras.
\end{example}

We summarize the results and examples; one will also deal with certain trivial extension algebras in the subsection \ref{Am2} (Proposition \ref{TEofAm2}).

\[\def\arraystretch{2}
\begin{array}{c||c|c|c}
 & \mbox{Positive Definite} & \mbox{Regular} & \mbox{Singular} \\\hline\hline
%\mbox{finite DEC} & 
%\vcenter{\hbox{$
% {\displaystyle
%  \mathop{
%   \mbox{TD}
%  }_{\mbox{\cite{EJR}}}
% } 
%$}}
%& 
%\vcenter{\hbox{${
% \displaystyle\mathop{
%  \ref{Kronecker}\ \mbox{(ID)}
% }_{\mbox{(not TC)}}
%}$}}
%& \vcenter{\hbox{$\displaystyle\mathop{\mbox{BGA of digon}}_{\mbox{(\ref{tdcBGA}, not TC)}}$}} \\\hline
\sup\delta(\II)<\infty &
\vcenter{\hbox{${
 \displaystyle\mathop{
  \ref{upperbound},\ \ref{TEofAm2}(m\leq3)
  }_{\mbox{(TD)}}
}$}}  &
\vcenter{\hbox{${
 \displaystyle\mathop{
  \ref{Kronecker}{\rm (3)}\ \mbox{(ID)}
 }_{\mbox{(not TC)}}
}$}}
&
\vcenter{\hbox{$
 \displaystyle\mathop{
  \mbox{\ref{tdcBGA} ($\supseteq$ \ref{Kronecker}{\rm (2)})}
  }_{
     \mbox{(not TC)}
    }
$}}
\\\hline
%& \mbox{(above)} & \mbox{(above)} \\\hline
\multirow{2}{*}{$\sup\delta(\II)=\infty$}
&
\vcenter{\hbox{$
 {\displaystyle
  \mathop{
    \ref{singularandnoupperbound}
  }_{\mbox{(TC or not?)}}
 }
$}}
&
\vcenter{\hbox{$
 {\displaystyle
  \mathop{
   \ref{singularandnoupperbound}\ (\mbox{ID})
  }_{\mbox{(TC or not?)}}
 }
$}}
&
\vcenter{\hbox{$
 {\displaystyle
  \mathop{
   \ref{singularandnoupperbound}
  }_{\mbox{(TC or not?)}}
 }
$}} \\\cline{2-4}
&
\vcenter{\hbox{$
 {\displaystyle
  \mathop{
    \ref{C3C3:C2}%,\ \ref{s}
  }_{\mbox{(TC or not?)}}
 }
$}}
&
\vcenter{\hbox{$
 {\displaystyle
  \mathop{
   \ref{TEofAm2}(m\geq5, \mbox{ID})
  }_{\mbox{(not TC)}}
 }
$}}   
&
\vcenter{\hbox{$
 {\displaystyle
  \mathop{
   \ref{TEofAm2}(m=4)
  }_{\mbox{(not TC)}}
 }
$}}   
\end{array}\]
Here, 
%PD $=$ Positive Definite, %PSD $=$ Positive Semi-Definite,
ID $=$ InDefinite, 
%DEC $=$ the Derived Equivalence Class, 
TD $=$ Tilting-Discrete, and TC $=$ Tilting-Connected.

\section{Examples of nonsymmetric algebras}

\subsection{}\label{ENSA1}
Let $A$ be the algebra as in Example \ref{TauTFwPDbutSID}:
\[\begin{array}{c@{\hspace{2cm}}c}
\vcenter{\xymatrix{
1 \ar[r]\ar[d]\ar[dr] & 2 \ar[d] \\
3 \ar[r] & 4
}}& \rad^2A=0
\end{array}\]
Recall that it is representation-infinite, $\tau$-tilting finite and has positive definite Cartan matrix; furthermore, it is of cyclotomic type.

We show that $A$ is not silting-discrete.
It is obtained that $A$ can be represented by the (nonpositive) dg path algebra of the graded quiver (with a certain differential) \cite{O}:
\[\xymatrix{
1 \ar[r]\ar[d]\ar[dr]\ar@{-->}@<4pt>[dr]\ar@{-->}@<-4pt>[dr] & 2 \ar[d] \\
3 \ar[r] & 4
}\]
Here, the dashed arrows are of degree $-1$.
Thus, the factor algebra $A/A(e_2+e_3)A$ is isomorphic to the dg path algebra $\B$ of the graded 3-Kronecker quiver $\xymatrix{1 \ar[r]\ar@{-->}@<4pt>[r]\ar@{-->}@<-4pt>[r] & 4}$, which is not silting-discrete by \cite[Proposition 2.25]{AH2}.
Thus, we derive from \cite[Corollary 2.18(1)]{AH2} that $A$ is not silting-discrete.

%\old{
%Next, we mention that $A$ is silting-connected.
%The two-time left mutations from $A$ with respect to $P_4$, and to $P:=P_2\oplus P_3$ yield the tilting complex $T$:
%\[\bigoplus\left\{
%\begin{array}{c}
%{\xymatrix@R=1mm{
% & P_1 \\
%P_4 \ar[r] & P_3 \\
%P_4 \ar[r] & P_2 \\
%P_4 \ar[r] & P
%}}
%\end{array}
%\right.\]
%It is easy that the endomorphism algebra of $T$ is the path algebra of extended Dynkin type $\AA_4$, which is silting-connected by \cite[Theorem 3.1]{AI}, and hence so is $A$.
%}

As above, we know that there exists a homological epimorphism from $A$ to $\B$.
Therefore, one has a recollement $(\D(\B), \D(A), \D(K^2))$ of the derived categories.
Moreover, one gets an isomorphism $\silt \B\simeq \silt_PA$ of posets,
where $\silt_PA$ stands for the subset of $\silt A$ consisting of silting complexes which has $P:=P_2\oplus P_3$ as a direct summand; see \cite{AH2}.

Let $B$ be the algebra as in Example \ref{TauTFvsPDC}: 
(we relabel the vertex 2 with 4)
\[\begin{array}{c@{\hspace{2cm}}c}
\vcenter{{\xymatrix{
1 \ar@<4pt>[r]^x\ar[r]|y & 4 \ar@<4pt>[l]^z
}}} & xz=0=yz
\end{array}\]
%Taking irreducible left mutations from $C$ twice (do first at $P_1$, and then do at $P_2$), we get the silting complex $\bigoplus\left\{
%\begin{array}{c}
%{\xymatrix@R=1mm{
%P_1 \ar[r]^z & P_2 \\
%P_1 &
%}}
%\end{array}
%\right.$
Taking the irreducible right mutation of $B$ with respect to $P_4$, we get the silting complex
\[T:=\bigoplus\left\{
\begin{array}{c}
{\xymatrix@R=1mm{
P_1 & \\
P_1 \ar[r]^z & P_4
}}
\end{array}
\right.\]
and one sees that its dg endomorhpism algebra is isomorphic to $\B$; see also \cite{LY}.
Thus, there exist isomorphisms and an injection of posets
\[\silt B \simeq \silt \B \simeq \silt_PA \hookrightarrow  \silt A.\]

This fact can be also obtained as follows; we avoid the dg algebra $\B$.
Let $U$ be the left mutation of $A$ with respect to $P_4$:
\[U:=\bigoplus\left\{
\begin{array}{c}
{\xymatrix@R=1mm{
 & P_1 \\
 & P_2 \\
 & P_3 \\
P_4 \ar[r] & P_1\oplus P_2 \oplus P_3
}}
\end{array}
\right.\]
This is a tilting complex; indeed, it is an APR tilting module.
We observe that the endomorphism algebra $A'$ of $U$ is presented by the following quiver with certain relations:
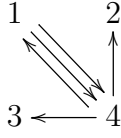
\[\xymatrix{
1 \ar@<4pt>[dr]\ar[dr] & 2 \\
3 & 4 \ar@<4pt>[ul] \ar[u]\ar[l]
}\]
One sees that the factor algebra $A'/A'(e_2+e_3)A'$ is isomorphic to $B$, and obtains isomorphisms and an injection of posets
\[\silt B\simeq \silt_{P'}A'\simeq \silt_PA\hookrightarrow \silt A.\]
We remark that this phenomenon comes from the following commutative diagram between derived categories:
\[\xymatrix@C=2cm{
\D(A) \ar[r]^{\mbox{{\tiny left mut.}}}_\sim & \D(A') \\
\D(\B) \ar[r]_{\mbox{{\tiny left mut.}}}^\sim \ar@{^{(}->}[u] & \D(B) \ar@{^{(}->}[u]
}\]

\subsection{}\label{subsec:Am}
Let us delve into the algebra $A_m$ defined in Example \ref{Am}:
\[\begin{array}{c@{\hspace{2cm}}c}
\vcenter{\xymatrix{
1 \ar@(lu,ld)_x\ar@{=>}[r]^y & 2
}} & x^m=0=xy
\end{array}\]
This is representation-finite if $\ell=1$, and otherwise it is representation-wild
\cite{BG, HM}.

We can reduce the power $m$ of $x$ so that the reduction preserves the silting quiver; the author really appreciates Florian Eisele informing of this fact.

\begin{proposition}
%Let $A_m$ be the algebra defined in Example \ref{Am}.
%Then 
We have an isomorphism $\silt A_{m+1}\simeq \silt A_2$ of posets.
\end{proposition}
\begin{proof}
Let $Q$ be the quiver $\xymatrix{1 \ar@(ur, ul)_x\ar@{=>}[r]^y \ar@(dr, dl)^{z_1}& 2 \ar@(dl,dr)_{z_2}}$ and put $\Gamma:=\widehat{KQ}/\overline{(xy, z_1y-yz_2, xz_1, z_1x)}$.
We set $X:=z_1+z_2-x^m$, which is in the center of $\Gamma$.
Furthermore, it is seen that $\Gamma$ is both free and finitely generated as a $K\dbracket{X}$-module; in fact, we have 
\[\Gamma=K\dbracket{X}e_1\oplus K\dbracket{X}x\oplus\cdots\oplus  K\dbracket{X}x^m\oplus K\dbracket{X}y_1\oplus\cdots \oplus K\dbracket{X}y_\ell\oplus K\dbracket{X}e_2.\]
Here, we distinguish the $\ell$ arrows $y$ by $y_i$'s.
So, $\Gamma$ is a $K\dbracket{X}$-order.

Now, we obtain isomorphisms
\[\Gamma/X\Gamma \simeq \widehat{KQ}/\overline{(xy, z_1y-yz_2, xz_1, z_1x, z_1-x^m, z_2)}
%\simeq \widehat{KQ^\circ}/\overline{(xy, x^my, x^{m+1})}
\simeq \widehat{KQ^\circ}/\overline{(xy, x^{m+1})} \simeq A_{m+1}.
\]
Here, $Q^\circ$ denotes the original quiver.
By \cite[Corollary 6.4]{E}, we have $\silt A_{m+1}\simeq \silt A_2$.
\end{proof}

Note that the silting quivers of the algebras $A_m$ for $m=1, 2$ are completely different.
For instance, let $\ell=1$.
Then, the silting quiver of $A_1$ consists of tetragons and pentagons, but that of $A_2$ admits hexagons.
Furthermore, we have the following observation.

\begin{proposition}\label{AmSID}
The algebra $A_m$ is silting-discrete if and only if $\ell=1=m$.
\end{proposition}
\begin{proof}
The `if' part is trivial. 
For the `only if' part, we only need to show that $A:=A_2$ for $\ell=1$ is not silting-discrete.
To make calculations easy, let us consider the algebra $B:=B_2$ for $\ell=1$ as in Example \ref{Am}, which is derived equivalent to $A$ and is a Nakayama algebra; recall that $B$ is as follows
\[\begin{array}{c@{\hspace{2cm}}c}
\vcenter{\xymatrix{
1 \ar@<2pt>[r]^z & 2 \ar@<2pt>[l]^y
}} & 
zyz=0
\end{array}\]
Now, we take the silting complex $T$ of $B$ got by four-times iterated irreducible left mutation from $B$ (the second once and then, the first thrice):
\[T:=\bigoplus\left\{
\begin{array}{c}
{\xymatrix@R=2mm{
P_2 \ar[r]^{yz} & P_2 \ar[r]^{\left(\begin{smallmatrix} z \\ 0 \end{smallmatrix}\right)} & {\begin{matrix} P_1 \\ P_2 \end{matrix}}\ar[r]^{\left(\begin{smallmatrix} zy & z \end{smallmatrix}\right)} & P_1 \\
 & & P_2 \ar[r]_z & P_1
}}
\end{array}
\right.\]
Note that $T$ is not a tilting complex.
By a direct calculation, we obtain that the (ordinary) endomorphism algebra of $T$ admits a multiple arrow; more presiesly, it is presented by the quiver $\xymatrix{1 \ar@<4pt>[r]^a & 2 \ar[l]|b\ar@<4pt>[l]^c}$
with relations $ca=0=aba=bab=bac$.
Thus, it turns out that $B$ is not $\tau$-tilting finite, whence $A$ is not silting-discrete.
\end{proof}

This fact seems a bit surprising. Actually, the algebra presented by the same quiver as that of $A_m$ for $\ell=1$ with $x^2=0$ (i.e., the monomial $xy$ does not vanish) is silting-discrete.
Furthermore, we obtain the following observation.

\begin{proposition}\label{Am2}
Let $m\geq1$ and $A$ be the algebra presented by the quiver with relation:
\[\begin{array}{c@{\hspace{2cm}}c}
\vcenter{\xymatrix{
1 \ar@(lu,ld)_x\ar[r]^y & 2
}} & x^m=0
\end{array}\]
Then the following are equivalent:
\begin{enumerate}
\item $A$ is silting-discrete;
\item It is $\tau$-tilting finite;
\item It is representation-finite;
\item $m\leq3$.
\end{enumerate}
\end{proposition}
\begin{proof}
For the representation type of $A$, we refer to \cite{BG, HM}.

Taking the irreducible left mutation of $A$ with respect to $P_2$ (it is a tilting complex),
we observe that its endomorphism algebra $B$ is the opposite algebra of $A$:
\[\begin{array}{c@{\hspace{2cm}}c}
\vcenter{\xymatrix{
1 \ar@(lu,ld)_x & 2 \ar[l]^y
}} & x^m=0
\end{array}\]
Moreover, the irreducible left mutation of $B$ with respect to $P_1$, which is also a tilting complex, possesses $A$ as the endomorphism algebra.
This implies that the twice irreducible left mutation of $A$ with respect to the second and first summands admits $g$-matrix
$\G:=\begin{pmatrix} 1 & 1 \\ 0 & -1 \end{pmatrix}\begin{pmatrix} -1 & 0 \\ m & 1 \end{pmatrix}=\begin{pmatrix} m-1 & 1 \\ -m & -1\end{pmatrix}$.
If the 1-shift of $A$ belongs to this repeated sequence (i.e., iterated left mutations by 2nd, 1st, 2nd, $\cdots$), then we must obtain an equality either $\G^s=-E$ or $\G^s\cdot\begin{pmatrix} 1 & 1 \\ 0 & -1 \end{pmatrix}=\begin{pmatrix} 0 & -1 \\ -1 & 0 \end{pmatrix}$ for some $s\geq0$.
It is easy to verify that it is impossible if $m\geq4$.
Thus, we find out that $A$ is $\tau$-tilting infinite; so it is silting-indiscrete.

It only remains to check that $A$ is silting-discrete when $m=3$.
By hand, we observe that $A$ is $\tau$-tilting finite.
Furthermore, it is seen that the irreducible left mutations of $A$ and $B$ respectively with respect to $P_1$ and $P_2$ have semisimple endomorphism algebras.
Therefore, it turns out that all iterated irreducible left mutations starting from $A$ admit $\tau$-tilting finite endomorphism algebras, whence $A$ is silting-discrete.
\end{proof}

Note that for $m=1, 2$ and $3$, the characteristic polynomial of $\G$ in the proof is cyclotomic of 3rd, 4th and 6th, respectively.
This corresponds that there are mutation paths of the length 3, 4 and 6 from $A$ to the 1-shift of $A$, respectively.

Related to Section \ref{sec:TEA}, let us discuss the trivial extension $\Lambda$ of the algebra $A$ as in Proposition \ref{Am2};
the Cartan matrix of $\Lambda$ is $\begin{pmatrix} 2m & m \\ m & 2\end{pmatrix}$, positive definite iff $1\leq m\leq 3$, non-positive definite positive semidefinite iff $m=4$ and indefinite otherwise.
We remark that $\Lambda$ is also the trivial extension of $A^\op$.
Then, one observes that the irreducible left mutations of $\Lambda$ with respect to $P_1$ and $P_2$ come from those of $A^\op$ and $A$ by applying the functor $-\otimes_A\Lambda$, respectively.
Furthermore, their endomorphism algebras are just $\Lambda$.
Thus, a similar argument as in the proof of Proposition \ref{Am2} works for $\Lambda$, and we obtain the following result.

\begin{proposition}\label{TEofAm2}
Keep the setting above. Then the following are equivalent:
\begin{enumerate}
\item $\Lambda$ is tilting-discrete;
\item It is $\tau$-tilting finite;
\item It has a positive definite Cartan matrix;
\item It is tilting-connected;
\item $m\leq3$.
\end{enumerate}
In this case, any algebra derived equivalent to $\Lambda$ is isomorphic to $\Lambda$ itself.
\end{proposition}

\if0
\old{
We obtain the following corollary.

\begin{corollary}
The dg path algebra of the graded $\infty$-Kronecker quiver with precisely one arrow of degree $-\ell$ for each $\ell\geq0$ is silting-indiscrete.
\end{corollary}
\begin{proof}
Let $A$ be the algebra given by the following quiver  with relations $xy=y^2=yz=0$:
\[\xymatrix{
1 \ar[r]^x & 2 \ar@(ul, ur)^y\ar[r]^z & 3
}\]
It follows from Proposition \ref{AmSID} and \cite[Corollary 2.18]{AH2} that $A$ is not silting-discrete.
Moreover, the idempotent truncation at the vertex 2 yields the dg path algebra $B$ of the graded $\infty$-Kronecker quiver with just one arrow of degree $-\ell$ for each $\ell\geq2$; see \cite[Example 2.13]{CJS},
which is silting-indiscrete by \cite[Corollary 2.18]{AH2} ... converse...
\end{proof}
}
\fi

\subsection{}
We compare 2-term and 3-term silting complexes.
Let $A$ be an algebra.
For a silting complex $T$ and $d>0$, we set
$\dsilt{d}{T}A:=\{U\in\silt A\ |\ T\geq U\geq T[d-1] \}$.
Trivially, $\dsilt{1}{T}A=\{T\}$ and $\dsilt{2}{A}A$ and $\dsilt{3}{A}A$ consist of 2-term and 3-term silting complexes, respectively.
Then we obtain the following result.

\begin{proposition}\label{2vs3}
Let $A$ be a $\tau$-tilting finite algebra and $d>0$.
Then we have the equality
$\dsilt{(d+1)}{A}A=\bigcup_{T}\dsilt{d}{T}A$, where $T$ runs in $\dsilt{2}{A}A$.
Moreover, any 3-term presilting complex is partial silting.
Further, $\dsilt{3}{A}A$ is a finite set if and only if any 2-term silting complex has $\tau$-tilting finite endomorphism algebra.
\end{proposition}
\begin{proof}
The first assertion follows from \cite[Proposition 2.8 and 2.10]{AM}.
The second one is due to \cite[Proposition 2.16]{Ai}.
The last one is by \cite{AIR}.
\end{proof}

\section*{Acknowledgements}
The author is very grateful to Florian Eisele and Wassilij Gnedin for useful discussions. 

\appendix
\section{Algebras of (generalized) cyclotomic type}\label{append:cyclotomic}

We briefly explain advantages of algebras admitting positive definite Cartan matrices; 
for instance, we saw such algebaras appear if thier trivial extensions are $\tau$-tilting finite.

Let $A$ be an algebra with nonsingular Cartan matrix $C$ and $\Phi$ denote its Coxeter matrix: i.e., $\Phi:=-C^\tr C^{-1}$.
Here, $(-)^\tr$ stands for the transpose of the matrix.
We say that $A$ is \emph{of (generalized) cyclotomic type} if the characteristic polynomial of $\Phi$ factorizes as the product of cyclotomic polynomials (all the eigenvalues of $\Phi$ have abusolute value one).
Note that if an algebra of  generalized cyclotomic type has an integral Coxeter matrix (e.g., $A$ is Iwanaga--Gorenstein), then it is of cyclotomic type, due to Kronecker's theorem.

\begin{remark}\label{Defofcyclotomic}
There is another definition of Coxeter matrices, which does not need the assumption of $A$ having a regular Cartan matrix.
However, we assume that $A$ has finite right injective dimension.
Then, the triangle functor $\tau:=-\Ltensor_A DA[-1]:\Db(\mod A)\to\Db(\mod A)$ is well-defined, where $D$ stands for the $K$-dual.
Now, we define the \emph{Coxeter matrix} as the matrix expression of $[\tau]: K_0(\Db(\mod A))\to K_0(\Db(\mod A))$ with respect to the canonical basis $\{S_i\}$.
Here, $S_i$ is a simple module corresponding to $P_i$.
%If $A$ is Iwanaga--Gorenstein, then $\tau$ induces the Auslander--Reiten translation of $\Kb(\proj A)$ and 
If in addition the Cartan matrix $C$ of $A$ is regular, then the Coxeter matrix coincides with the matrix $\Phi$.
Following this manner, every gentle algebra is of cyclotomic type \cite[Theorem 6.14]{JGM}.
Also, any selfinjective algebra has cyclotomic type.

To illustrate, we use the radical-square-zero algebra presented by the quiver $\xymatrix{1 \ar@<2pt>[r] & 2 \ar@<2pt>[l]}$.
Then, its Cartan and Coxeter matrices are $\begin{pmatrix} 1 & 1 \\ 1 & 1\end{pmatrix}$ and $\begin{pmatrix} 0 & -1 \\ -1 & 0\end{pmatrix}$, respectively.
Thus, this says that the algebra is actually of cyclotomic type with singular Cartan matrix.

%When $A$ is not Iwanaga--Gorenstein, we should distinguish the two definitions of Coxter matrices; the latter Coxeter matrix is always integral.
%In this paper, we adapt the former definition, because we are interested in the case that Cartan matrices are positive definite.
\end{remark}

The following is the main result of this appendix we might as well recall.

\begin{proposition}\label{cyclotomic}
Let $A$ be an algebra with positive definite Cartan matrix $C$.
Then:
\begin{enumerate}
%\item $A$ is $\tau$-tilting finite;
%\item For the Cartan matrix $C$ of $A$, the quadratic form $\overline{x^t}Cx$ has a strictly positive real part, where $\overline{(-)}$ denotes the complex conjugate.
%\item For the Cartan matrix $C$ of $A$, the quadratic form $x^tCx$ is positive definite.
%In particular, $C$ is invertible;
\item $A$ is of generalized cyclotomic type;
\item The Coxeter matrix $\Phi$ of $A$ is diagonalizable;
%\item Its all eigenvalues have absolute value one;
\item It has no eigenvalue one.
\end{enumerate}
\end{proposition}

Although we do not assume that $A$ is triangular, the discussion as in \cite{P} works for the assertions; see also \cite{S}.
For the convenience of the reader, we give a proof here.

\begin{proof}
%The algebra $A$ is $\tau$-tilting finite by \cite[Theorem 5.12(d)]{DIRRT}.
%Since the Cartan matrix of $\Lambda$ is given by $C+C^t$ and it is positive definite by \cite[Proposition 3.2]{H},
%we observe that $\overline{x^t}Cx$ admits a strictly positive real part for any nonzero (complex) vector $x$.
%So, $C$ is invertible.
%
We consider the quadratic form $\phi(x):=\overline{x^\tr}C^{-1}x$, where $(-)^\tr$ and $(-)$ denote the transpose and the complex conjugate of a matrix.
Note that its real part is strictly positive.
Indeed, putting $y:=Cx$ we get
$\phi(y)=\overline{x^\tr}C^tx=\overline{\overline{x^\tr}Cx}$.
When $x$ runs all vectors, so does $y$.
Hence, we figure out the claim since the real part of $\overline{x^\tr}Cx$ is positive.

Let $\lambda$ be an eigenvalue of $\Phi$ and $x$ the corresponding eigenvector.
We have equalities
\[\phi(x)=-\overline{x^\tr}C^{-\tr}\Phi x=-\lambda\overline{x^\tr}C^{-\tr}x=-\lambda\overline{\phi(x)}.\]
Therefore, we get $|\lambda|=1$, but $\lambda\neq1$.
Thus, the assertions (1) and (3) follow.

Suppose $\Phi$ has a generalized eigenvector $y$ of rank 2 with corresponding eigenvalue $\lambda$, and put $x:=(\Phi-\lambda E)y\ (\neq0)$; $E$ is the identity matrix.
Then, we observe equalities
\[\phi(x)=\overline{(C^{-\tr}x)^\tr}\cdot x=\overline{(-C^{-1}y-\lambda C^{-\tr}y)^\tr}\cdot x=-\overline{y^\tr}C^{-\tr}x-\overline{\lambda}\cdot\overline{y^\tr}C^{-1}x=0.\]
The last equality holds because of $|\lambda|^2=1$.
This contradiction shows the assertion (2).
\end{proof}

\begin{example}
Let $A$ be a selfinjective algebra.
Then, the Coxeter matrix $\Phi$ is determined by the Nakayama permutation $\sigma$ and is diagonalizable.
Moreover, the Coxeter polynomial is the product of polynomials $x^\ell\pm1$, where $\ell$ is the length of a cyclic permutation in the cycle decomposition of $\sigma$ and $\pm$ is determined by the parity of $\ell$;
minus if $\ell$ is even and plus if $\ell$ is odd.
So, we see that $\Phi$ has no eigenvalue one iff $\sigma$ is an even permutation.
\end{example}

Let $A$ be an Iwanaga--Gorenstein algebra and put $\nu:=-\Ltensor_ADA$.
We say that $A$ is \emph{fractionally Calabi--Yau} (abbr. FCY) if there exist integers $p$ and $q>0$ such that $\nu^q$ and $[p]$ are functorially isomorphic on $\Kb(\proj A)$.
In the case, some power of $\tau:=\nu\circ[-1]$ becomes a shift; so, the Coxeter matrix $\Phi$ of $A$ is periodic; i.e., $\Phi^\ell=E$ for some $\ell>0$.

\begin{theorem}\cite{P}
An FCY algebra is of cyclotomic type with $\Phi$ diagonalizable.
\end{theorem}

Furthermore, the trivial extension of an FCY algebra admits nice structure.

\begin{theorem}\cite{CDIM}
An algebra of finite global dimension is FCY if and only if its trivial extension is periodic; i.e., every module is periodic by the syzygy.
\end{theorem}

\section{When is a $\tau$-tilting finite algebra with positive definite Cartan matrix piecewise hereditary of Dynkin type?}\label{append:phad}

We recall that the (Morita equivalence) class $\H$ consists of $\tau$-tilting infinite algebras, algebras with non-positive definite Cartan matrices, and PHADs.
That is, if a member of $\H$ is $\tau$-tilting finite and has a positive definite Cartan matrix, then it is a PHAD.
In this appendix, we show that $\H$ contains:
\begin{enumerate}
\setcounter{enumi}{-1}
\item piecewise hereditary algebras;
\item simply-connected algebras;
\item algebras with preprojective components whose Coxeter matrices have trace $-1$;
\item locally hereditary algebras;
\item algebras admitting a sincere and directing indecomposable module;
\item algebras whose trivial extensions are stable equivalent to the trivial extensions of some derived hereditary algebras.
\end{enumerate}
We also exhibit that:
\begin{enumerate}
\setcounter{enumi}{5}
\item it is closed under taking one-point extensions by indecomposable modules.
\end{enumerate}

In the case of (0), the algebra is $\tau$-tilting finite if and only if it has a positive definite Cartan matrix if and only if it is of Dynkin type.
The case (2) follows from \cite{AS2, W}.

Let us check above in order from (2).

\begin{proposition}
Every algebra with preprojective component and whose Coxeter matrix admits trace $-1$ is a member of $\H$.
\end{proposition}
\begin{proof}
Let $A$ be $\tau$-tilting finite and have a positive definite Cartan matrix.
It is well-known that any preprojective module is a brick.
Since $A$ is $\tau$-tilting finite, we obtain from \cite[Theorem 4.2]{DIJ} that the preprojective component coincides with the whole Auslander--Reiten quiver; in particular, $A$ is representation-directed \cite{Mo}.
This implies that the $i$th Hochschild cohomology $HH^i(A)$ vanishes for any $i\geq2$ \cite[Corollary 5.4]{Ha1}.
Furthermore, thanks to Happel's criteria \cite{Ha2}, we have $HH^1(A)=0$ because the trace of the Coxeter matrix is $-1$.
Thus, it turns out that $A$ is simply-connected \cite[Corollary 5.5]{Ha1}; done.
\end{proof}

%\begin{remark}
%In this proposition, we do not know if the assmption of the Coxeter matrix having trace $-1$ can be dropped.
%That is, one asks whether or not, a representation-directed algebra with positive definite Cartan matrix is a PHAD.
%More generally, if $A$ is a $\tau$-tilting finite triangular algebra (i.e., the Gabriel quiver is acyclic) with positive definite Cartan matrix, then does its Coxeter matrix have trace $-1$? (the triangle-ness is needed.)
%\end{remark}

We call an algebra \emph{locally hereditary} provided any homomorphism between indecomposable projective modules is zero or a monomorphism.

%A similar discussion works for locally hereditary algebras; see \cite[Theorem 4.11]{AHMW}.

%\begin{corollary}
%The following conditions are equivalent for a locally hereditary algebra $A$;
%\begin{enumerate}
%\item The trivial extension of $A$ is tilting-discrete;
%\item It is $\tau$-tilting finite;
%\item $A$ is a piecewise hereditary algebra of Dynkin type.
%\end{enumerate}
%\end{corollary}
\begin{proposition}
All locally hereditary algebras are in $\H$.
\end{proposition}
\begin{proof}
By \cite[Theorem 4.11]{AHMW} and its proof, a $\tau$-tilting finite locally hereditary algebra is strongly simply-connected.
\end{proof}

%Since the existence of a sincere and directing indecomposable module gives a strong influence to the algebra, we have the following observation.

%\begin{corollary}
%Assume that $A$ possesses a sincere and directing indecomposable module.
%Then the following are equivalent:
%\begin{enumerate}
%\item The trivial extension of $A$ is tilting-discrete;
%\item It is $\tau$-tilting finite;
%\item The Euler form of $A$ is positive definite; 
%\item $A$ is a tilted algebra of Dynkin type.
%\end{enumerate}
%\end{corollary}
\begin{proposition}\label{sinceredirecting}
The class of algebras possessing a sincere and directing indecomposable module is contained in $\H$.
\end{proposition}
\begin{proof}
By \cite[IX, Theorem 2.6(b)]{ASS}, the algebra is tilted and has finite global dimension.
If it admits a positive definite Cartan matrix ($=$ Euler form), then it is of Dynkin type.
\end{proof}

\begin{proposition}\label{stableequivalent}
An algebra whose trivial extension is stable equivalent to the trivial extension of some derived hereditary algebra belongs to $\H$.
\end{proposition}
\begin{proof}
Let $\Lambda$ and $\Gamma$ be the trivial extensions of an algebra $A$ and a derived hereditary algebra $B$, respectively.
We assume that $\Lambda$ and $\Gamma$ are stable equivalent.
Since $B$ is derived equivalent to a hereditary algebra,
we observe that $\Gamma$ is derived equivalent to the trivial extension of the algebra.
As a derived equivalence between symmetric algebras gives rise to a stable equivalence between them, we can suppose that $B$ is hereditary.
It follows from \cite{PX} that $\Lambda$ is isomorphic to the trivial extension of an iterated tilted algebra (with the same type as $B$).
This finishes the proof.
\end{proof}

We remark by the proofs that if algebras as in Proposition \ref{sinceredirecting} and \ref{stableequivalent} have positive definite Cartan matrices, then they are PHADs; that is, in the cases, the $\tau$-tilting finiteness is automatically implied.

An algebra $B$ is said to be \emph{accessible} from an algebra $A$ if there exists a sequence $A=:B_0, B_1,\cdots, B_r:=B$ of algebras such that $B_{i+1}$ is the one-point (co)extension of $B_i$ by an exceptional $B_i$-module. 

\begin{proposition}
The one-point extension of an algebra $A$ in $\H$ by an indecomposable module $M$ is in $\H$.
Hence, any accessible algebra from an algebra in $\H$ belongs to $\H$.
\end{proposition}
\begin{proof}
Put $\widetilde{A}:=\begin{pmatrix} K & M \\ 0 & A\end{pmatrix}$ and assume that it is $\tau$-tilting finite with positive definite Cartan matrix.
Then, $A$ is isomorphic to $e\widetilde{A}e$ for a certain idempotent $e$ of $\widetilde{A}$, and so it is also $\tau$-tilting finite and has a positive definite Cartan matrix.
We derive from $A\in\H$ that $A$ is a PHAD.
As $M$ is indecomposable, applying the trick of Barot--Lenzing \cite[Theorem 1]{BL} leads to the fact that $\widetilde{A}$ is piecewise hereditary.
Since $\widetilde{A}$ has a positive definite Cartan matrix, it is a PHAD.
\end{proof}

Finally, let us give an analogue of this proposition.

\begin{proposition}
Let $A$ be a triangular algebra with a source or sink $e$ such that all summands of $A/AeA$ are in $\H$ and:
\begin{enumerate}[$(i)$]
\item $eA$ has a separated radical if $e$ is a source;
\item $eDA$ has a separated socle factor if $e$ is a sink.
\end{enumerate}
Then $A$ belongs to $\H$.
\end{proposition}
\begin{proof}
We only show the case (ii); the other can be handled similarly.

Write $A/AeA={\displaystyle\prod_{i=1}^r A_i}$ with $A_i$ ring-indecomposable and  $eDA/\soc(eDA)={\displaystyle\bigoplus_{i=1}^rM_i}$ such that $M_i$ is an $A_i$-module.
Since $eDA$ has a separated socle factor, we see that $M_i$ is indecomposable.
Now, let us assume that $A$ is $\tau$-tilting fintie and admits a positive definite Cartan matrix.
Then, $A/AeA\ (\simeq (1-e)A(1-e))$ is also $\tau$-tilting finite and has a positive definite Cartan matrix, whence so are $A_i$'s, and they are PHADs by assumption.

Let $B$ be the $\nu$-reflection at the sink; i.e., the factor algebra of the one-point extension $\begin{pmatrix} K & eDA \\ 0 & A\end{pmatrix}$ by the ideal $\left\langle\begin{pmatrix} 0 & 0 \\ 0 & e\end{pmatrix}\right\rangle$.
As the discussion above, we figure out that $B$ is isomorphic to the upper triangular matrix algebra
\[\begin{pmatrix}
K & M_1 & M_2 & \cdots & M_r \\
& A_1 &  &  &   \\
& & A_2 &  &  \\
&&& \ddots & \\
&&&& A_r
\end{pmatrix}.\]
Applying the trick of Barot--Lenzing \cite[Theorem 1]{BL}, it turns out that $B$ is piecewise hereditary.
Moreover, it follows from \cite{TW1} that $A$ and $B$ are derived equivalent, and so their Cartan matrices represent equivalent quadratic forms.
This implies that $B$ also has a positive definite Cartan matrix,
whence it is a PHAD; and, so is $A$.
Thus, $A\in\H$.
\end{proof}

%%%%%%%%%%%%%%%%%%%%%%%%%%%%%%%%%%%%%%%%%%%%%%%%%%%%%%%%

%%%%%%%%%%%%%%%%%%%%%%%%%%%%%%%%%%%%%%%%%%%%%%%%%%%%%%%%

\begin{thebibliography}{AHMW}

\bibitem[Ad]{Ad1}
{\sc T. Adachi},
Characterizing $\tau$-tilting finite algebras with radical square zero.
{\it Proc. Amer. Math. Soc.} {\bf 144} (2016), no. 11, 4673--4685.

%\bibitem[A2]{Ad}
%{\sc T. Adachi},
%The classification of $\tau$-tilting modules over Nakayama algebras.
%{\it J. Algebra} {\bf 452} (2016), 227--262.

\bibitem[AAC]{AAC}
{\sc T. Adachi, T. Aihara and A. Chan},
Classification of two-term tilting complexes over Brauer graph algebras.
{\it Math. Z.} {\bf 290} (2018), no. 1--2, 1--36.

%\bibitem[AK]{AdK}
%{\sc T. Adachi and R. Kase},
%Examples of tilting-discrete self-injective algebras which are not silting-discrete.
%Preprint (2020), arXiv: 2012.14119.

\bibitem[AIR]{AIR}
{\sc T. Adachi, O. Iyama and I. Reiten},
$\tau$-tilting theory.
{\it Compos. Math.} {\bf 150}, no. 3, 415--452 (2014).

%\bibitem[AMY]{AMY}
%{\sc T. Adachi, Y. Mizuno and D. Yang},
%Discreteness of silting objects and $t$-structures in triangulated categories.
%{\it Proc. Lond. Math. Soc. (3)} {\bf 118} (2019), no. 1, 1--42.

\bibitem[Ai1]{Ai}
{\sc T. Aihara},
Tilting-connected symmetric algebras.
{\it Algebr. Represent. Theory} {\bf 16} (2013), no. 3, 873--894.

%\bibitem[Ai2]{Ai2}
%{\sc T. Aihara},
%On silting-discrete triangulated categories.
%{\it Proceedings of the 47th Symposium on Ring Theory and Representation Theory}, 7--13, {\it Symp. Ring Theory Represent. Theory Organ. Comm., Okayama}, 2015.

%\bibitem[AGI]{AGI}
%{\sc T. Aihara, J. Grant and O. Iyama},
%Private communication.

%\bibitem[AiK]{AK}
%{\sc T. Aihara and R. Kase},
%Algebras sharing the same support $\tau$-tilting poset with tree quiver algebras.
%{\it Q. J. Math.} {\bf 69} (2018), no. 4, 1303--1325.

\bibitem[AH1]{AH}
{\sc T. Aihara and T. Honma},
$\tau$-tilting finite triangular matrix algebras.
{\it J. Pure Appl. Algebra} {\bf 225} (2021), no. 12, Paper No. 106785, 10pp.

\bibitem[AH2]{AH2}
{\sc T. Aihara and T. Honma},
When is the silting-discreteness inherited?
{\it Nagoya Math. J.} {\bf 256} (2024), 905--927.

\bibitem[AHMW]{AHMW}
{\sc T. Aihara, T. Honma, K. Miyamoto and Q. Wang},
Report on the finiteness of silting objects.
{\it Proc. Edinb. Math. Soc. (2)} {\bf 64} (2021), no. 2, 217--233.

\bibitem[AI]{AI}
{\sc T. Aihara and O. Iyama},
Silting mutation in triangulated categories.
{\it J. Lond. Math. Soc. (2)} {\bf 85} (2012), no. 3, 633--668.

\bibitem[AM]{AM}
{\sc T. Aihara and Y. Mizuno},
Classifying tilting complexes over preprojective algebras of Dynkin type.
{\it Algebra Number Theory} {\bf 11} (2017), no. 6, 1287--1315.

%\bibitem[ANR]{ANR}
%{\sc S. Al-Nofayee and J. Rickard},
%Rigidity of tilting complexes and derived equivalence for self-injective algebras.
%\url{http://www.maths.bris.ac.uk/~majcr/papers.html} (preprint)

%\bibitem[AHMV]{AHMV}
%{\sc L. Angeleri-Hugel, F. Marks and J. Vitoria},
%Partial silting objects and smashing subcategories.
%{\it Math. Z.} {\bf 296} (2020), no. 3--4, 887--900.

\bibitem[An]{An}
{\sc M.A. Antipov},
Derived equivalence of symmetric special biserial algebras.
{\it J. Math. Sci. (N.Y.)} {\bf 147} (2007), no. 5, 6981--6994.

\bibitem[An2]{An2}
{\sc M. A. Antipov},
The structure of the stable Grothendieck group of a symmetric SB-algebra.
translation in {\it J. Math. Sci. (N.Y.)} {\bf 161} (2009), no. 4, 474--482.


%\bibitem[ALPP]{ALPP}
%{\sc K. K. Arnesen, R. Laking, D. Pauksztello and M. Prest},
%The Ziegler spectrum for derived-discrete algebras.
%{\it Adv. Math.} {\bf 319} (2017), 653--698.

\bibitem[As]{As}
{\sc I. Assem},
A course on cluster tilted algebras.
{\it Homological methods, representation theory, and cluster algebras},
127--176, CRM Short Courses, {\it Springer, Cham}, 2018.

\bibitem[ABCJP]{ABCJP}
{\sc I. Assem, T. Brustle, G. Charbonneau-Jodoin and P.-G. Plamondon},
Gentle algebras arising from surface triangulations.
{\it Algebra Number Theory} {\bf 4} (2010), no. 2, 201--229.

\bibitem[ABS]{ABS}
{\sc I. Assem, T. Brustle and R. Schiffler},
Cluster-tilted algebras as trivial extensions.
{\it Bull. Lond. Math. Soc.} {\bf 40} (2008), no. 1, 151--162.

\bibitem[ANS]{ANS}
{\sc I. Assem, J. Nehring and A. Skowronski},
Domestic trivial extensions of simply connected algebras.
{\it Tsukuba J. Math.} {\bf 13} (1989), no. 1, 31--72.

\bibitem[ASS]{ASS}
{\sc I. Assem, D. Simson and A. Skowronski},
Elements of the representation theory of associative algebras. Vol. 1.
Techniques of representation theory.
London Mathematical Society Student Texts, {\bf 65}.
{\it Cambridge University Press, Cambridge}, 2006.

\bibitem[AS1]{AS1}
{\sc I. Assem and A. Skowronski},
Iterated tilted algebras of type $\widetilde{A_n}$.
{\it Math. Z.} {\bf 195} (1987), no. 2, 269--290.

\bibitem[AS2]{AS2}
{\sc I. Assem and A. Skowronski},
Quadratic forms and iterated tilted algebras.
{\it J. Algebra} {\bf 128} (1990), no. 1, 55--85.

%\bibitem[AHR]{AHR}
%{\sc I. Assem, D. Happel and O. Roldan},
%Representation-finite trivial extension algebras.
%{\it. J. Pure Appl. Algebra} {\bf 33} (1984), no. 3, 235--242.

%\bibitem[Au]{Au}
%{\sc J. August},
%On the finiteness of the derived equivalence classes of some stable endomorphism rings.
%{\it Math. Z.} {\bf 296} (2020), no. 3--4, 1157--1183.

\bibitem[AD]{AD}
{\sc J. August and A. Dugas},
Silting and tilting for weakly symmetric algebras.
{\it Algebr. Represent. Theory} {\bf 26} (2023), no. 1, 169--179.

%\bibitem[AR]{AR}
%{\sc M. Auslander and I. Reiten},
%On the representation type of triangular matrix rings.
%{\it J. London Math. Soc. (2)} {\bf 12} (1975/76), no. 3, 371--382

%\bibitem[ARS]{ARS}
%{\sc M. Auslander, I. Reiten and S. O. Smalo},
%Representation theory of Artin algebras.
%Cambridge Studies in Advanced Mathematics, {\bf 36}.
%{\it Cambridge University Press, Cambridge}, 1995.

%\bibitem[BS]{BS}
%{\sc P. Balmer and M. Schlichting},
%Idempotent completion of triangulated categories.
%{\it J. Algebra} {\bf 236} (2001), no. 2, 819--834.

\bibitem[BL]{BL}
{\sc M. Barot and H. Lenzing},
One-point extensions and derived equivalence.
{\it J. Algebra} {\bf 264} (2003), no. 1, 1--5.

%\bibitem[B]{B}
%{\sc J. Bastian},
%Derived equivalences of cluster-tilted algebras of type $\widetilde{A}_n$ and $D_n$.
%PhD thesis, Leibniz University Hannover (2011), 157pp.
%URL: \url{https://repo.uni-hannover.de/items/0cba6565-e638-4986-a02c-bb8b5a1be270}.

\bibitem[BHL]{BHL}
{\sc J. Bastian, T. Holm and S. Ladkani},
Derived equivalence classification of the cluster-tilted algebras of Dinkin type $E$.
{\it Algebr. Represent. Theory} {\bf 16} (2013), no. 2, 527--551.

\bibitem[BHL2]{BHL2}
{\sc J. Bastian, T. Holm and S. Ladkani},
Towards derived equivalence classification of the cluster-tilted algebras of Dynkin type $D$.
{\it J. Algebra} {\bf 410} (2014), 277--332.

%\bibitem[BM]{BM}
%{\sc V. Bekkert and H. A. Merklen},
%Indecomposables in derived categories of gentle algebras.
%{\it Algebr. Represent. Theory} {\bf 6} (2003), no. 3, 285--302.

%\bibitem[BS1]{BS1}
%{\sc J. Bialkowski and A. Skowronski},
%On tame weakly symmetric algebras having only periodic modules.
%{\it Arch. Math. (Basel)} {\bf 81} (2003), no. 2, 142--154.

%\bibitem[BS2]{BS2}
%{\sc J. Bialkowski and A. Skowronski},
%Socle deformations of self-injective algebras of tubular type.
%{\it J. Math. Soc. Japan} {\bf 56} (2004), no. 3, 687-716.

%\bibitem[BHS]{BHS}
%{\sc J. Bialkowski, T. Holm and A. Skowronski},
%Derived equivalences for tame weakly symmetric algebras having only periodic modules.
%{\it J. Algebra} {\bf 269} (2003), no. 2, 652--668.

\bibitem[BS]{BS}
{\sc R. Bocian and A. Skowronski},
Symmetric special biserial algebras of Euclidean type.
{\it Colloq. Math.} {\bf 96} (2003), no. 1, 121--148.

\bibitem[BGS]{BGS}
{\sc G. Bobinski, C. Geiss and A. Skowronski},
Classification of discrete derived categories.
{\it Cent. Eur. J. Math.} {\bf 2} (2004), no. 1, 19--49.

\bibitem[BG]{BG}
{\sc K. Bongartz and P. Gabriel},
Covering spaces in representation-theory.
{\it Invent. Math.} {\bf 65} (1981/82), no. 3, 331--378.

%\bibitem[BPP]{BPP}
%{\sc N. Broomhead, D. Pauksztello and D. Ploog},
%Discrete derived categories II: the silting pairs CW complex and the stability manifold.
%{\it J. Lond. Math. Soc. (2)} {\bf 93} (2016), no. 2, 273--300.

%\bibitem[BY]{BY}
%{\sc T. Brustle and D. Yang},
%Ordered exchange graphs.
%{\it Advances in representation theory of algebras}, 135--193.
%EMS Ser. Congr. Rep., {\it Eur. Math. Soc., Zurich}, 2013.

\bibitem[BMR]{BMR}
{\sc A.B. Buan, R. J. Marsh and I. Reiten},
Cluster-tilted algebras.
{\it Trans. Amer. Math. Soc.} {\bf 359} (2007), no. 1, 323--332.

\bibitem[BMR2]{BMR2}
{\sc A.B. Buan, R. B. Marsh and I. Reiten},
Cluster mutation via quiver representations.
{\it Comment. Math. Helv.} {\bf 83} (2008), no. 1, 143--177.

\bibitem[BV]{BV}
{\sc A. B. Buan and D. F. Vatne},
Derived equivalence classification for cluster-tilted algebras of type $A_n$.
{\it J. Algebra} {\bf 319} (2008), no. 7, 2723--2738.

\bibitem[CDIM]{CDIM}
{\sc A. Chan, E. Darpo, O. Iyama and R. Marczizik},
Periodic trivial extension algebras and fractionally Calabi--Yau algebras.
{\it Ann. Sci. Ec. Norm. Super. (4) } {\bf 58} (2025), no. 2, 463--510.

%\bibitem[CKL]{CKL}
%{\sc A. Chan, S. Koenig and Y. Liu},
%Simple-minded system, configurations and mutations for representation-finite self-injective algebras.
%{\it J. Pure Appl. Algebra} {\bf 219} (2015), no. 6, 1940--1961.

\bibitem[CJS]{CJS}
{\sc W. Chang, H. Jin and S. Schroll},
Recollements of partially wrapped Fukaya categories and surface cuts.
Preprint (2022), arXiv: 2206.11196.

%\bibitem[CC]{CC}
%{\sc X. Chen and X.-W. Chen},
%An informal introduction to dg categories.
%Preprint (2019), arXiv: 1908.04599.

%\bibitem[CLZZ]{CLZZ}
%{\sc X.-W. Chen, Z.-W. Li, X. Zhang and Z. Zhao},
%A non-vanishing result on the singularity category.
%Preprint (2023), arXiv: 2301.01897.

%\bibitem[CH]{CH}
%{\sc F. U. Coelho and D. Happel},
%Quasitilted algebras admit a preprojective component.
%{\it Proc. Amer. Math. Soc.} {\bf 125} (1997), no. 5, 1283--1291.

\bibitem[DIJ]{DIJ}
{\sc L. Demonet, O. Iyama and G. Jasso},
$\tau$-tilting finite algebras, bricks and $g$-vectors.
{\it Int. Math. Res. Not. IMRN} 2019, no. 3, 852--892.

\bibitem[DIRRT]{DIRRT}
{\sc L. Demonet, O. Iyama, N. Reading, I. Reiten and H. Thomas},
Lattice theory of torsion classes.
Preprint (2017), arXiv: 1711.01785.

\bibitem[E]{E}
{\sc F. Eisele},
Bijections of silting complexes and derived Picard groups.
{\it J. Lond. Math. Soc. (2)} {\bf 106}, (2022), no. 2, 1008--1060.

\bibitem[EJR]{EJR}
{\sc F. Eisele, G. Janssens and T. Raedschelders},
A reduction theorem for $\tau$-rigid modules.
{\it Math. Z.} {\bf 290} (2018), no. 3--4, 1377--1413.

%\bibitem[EH]{EH}
%{\sc E. G. Escolar and Y. Hiraoka},
%Persistence modules on commutative ladders of finite type.
%{\it Discrete Comput. Geom.} {\bf 55} (2016), no. 1, 100--157.

%\bibitem[FP]{FP}
%{\sc E. A. Fernandez and M. I. Platzeck},
%Presentations of trivial extensions of finite dimensional algebras and a theorem of Sheila Benner.
%{\it J. Algebra} {\bf 249} (2002), no. 2, 326--344.


\bibitem[G]{G}
{\sc A. Grothendieck},
Groupes de classes des categories abeliennes et triangulees. complexes parfaits. (French)
{\it Cohomologie l-adique et Fonctions L}, 351--371, Lecture Notes in Math., {\bf 589}, {\it Springer, Berlin}, 1977.

%\bibitem[GL]{GL}
%{\sc W. Geigle and H. Lenzing},
%Perpendicular categories with applications to representations and sheaves.
%{\it J. Algebra} {\bf 144} (1991), 273--343.

%\bibitem[H]{H}
%{\sc D. Happel},
%Tilting sets on cylinders.
%{\it Proc. London Math. Soc. (3)} {\bf 51} (1985), no. 1, 21--55.

\bibitem[Ha1]{Ha1}
{\sc D. Happel},
Hochschild cohomology of finite-dimensional algebras.
{\it S\'eminaire d'Alg\`ebre Paul Dubreil et Marie-Paul Malliavin, 39\`eme Ann\'ee (Paris, 1987/1988)}, 108--126, Lecture Notes in Math., {\bf 1404}, {\it Springer, Berlin}, 1989.

\bibitem[Ha2]{Ha2}
{\sc D. Happel},
Triangulated categories in the representation theory of finite-dimensional algebras.
London Mathematical Society Lecture Note Series, {\bf 119}.
{\it Cambridge University Press, Cambridge}, 1998.



%\bibitem[HS]{HS}
%{\sc D. Happel and U. Seidel},
%Piecewise hereditary Nakayama algebras.
%{\it Algebr. Represent. Theory} {\bf 13} (2010), no. 6, 693--704.

%\bibitem[HV]{HV}
%{\sc D. Happel and D. Vossieck},
%Minimal algebras of infinite representation type with preprojective component.
%{\it Manuscripta Math.} {\bf 42} (1983), no. 2--3, 221--243.

\bibitem[Hi]{H}
{\sc N. Hiramae},
$\tau$-Tilting finiteness of group algebras over generalized symmetric groups.
Preprint (2024), arXiv: 2405.10726.

\bibitem[HM]{HM}
{\sc M. Hoshino and J. Miyachi},
Tame two-point algebras.
{\it Tsukuba J. Math.} {\bf 12} (1988), no. 1, 65--96.

%\bibitem[HK]{HK}
%{\sc Z. Hua and B. Keller},
%Cluster categories and rational curves.
%Preprint (2020), arXiv: 1810.00749.

%\bibitem[HZS]{HZS}
%{\sc B. Huisgen-Zimmermann and M. Saorin},
%Geometry of chain complexes and outer automorphisms under derived equivalence.
%{\it Trans. Amer. Math. Soc.} {\bf 353} (2001), no. 12, 4757--4777.

%\bibitem[IY]{IY}
%{\sc O. Iyama and D. Yang},
%Silting reduction and Calabi-Yau reduction of triangulated categories.
%{\it Trans. Amer. Math. Soc.} {\bf 370} (2018), no. 11, 7861--7898.

%\bibitem[IX]{IX}
%{\sc O. Iyama and Z. Xiaojin},
%Tilting modules over Auslander--Gorenstein algebras.
%{\it Pacific J. Math.} {\bf 298} (2019), no. 2, 399--416.

\bibitem[JGM]{JGM}
{\sc J. A. Jimenez Gonzalez and A. Mroz},
A graph theoretic model for the derived categories of gentle algebras and their homological bilinear forms.
 Preprint (2024), arXiv: 2407.04817.

%\bibitem[JSW]{JSW}
%{\sc H. Jin, S. Schroll and Z. Wang},
%A complete derived invariant and silting theory for graded gentle algebras.
%Preprint (2023), arXiv: 2303.17474.

%\bibitem[KY1]{KY}
%{\sc M. Kalck and D. Yang},
%Relative singularity categories I: Auslander resolutions.
%{\it Adv. Math.} {\bf 301} (2016), 973--1021.

%\bibitem[KY2]{KY3}
%{\sc M. Kalck and D. Yang},
%Derived categories of graded gentle one-cycle algebras.
%{\it J. Pure Appl. Algebra} {\bf 222} (2018), no. 10, 3005--3035.

%\bibitem[KY3]{KY2}
%{\sc M. Kalck and D. Yang},
%Relative singularity categories II: dg modules.
%Preprint (2018), arXiv: 1803.08192.

%\bibitem[K]{K}
%{\sc B. Keller},
%Deriving DG categories.
%{\it Ann. Sci. Ecole Norm. Sup.} (4) {\bf 27} (1994), 63--102.

%\bibitem[KM]{KM}
%{\sc Y. Kimura and Y. Mizuno},
%Two-term tilting complexes for preprojective algebras of non-Dynkin type.
%{\it Comm. Algebra} {\bf 50} (2022), no. 2, 556--570.

%\bibitem[KoY]{KoY}
%{\sc S. Koenig and D. Yang},
%Silting objects, simple-minded collections, $t$-structures and co-$t$-structures for finite-dimensional algebras.
%{\it Doc. Math.} {\bf 19} (2014), 403--438.

%\bibitem[L]{L}
%{\sc S. Ladkani},
%On derived equivalences of lines, rectangles and triangles.
%{\it J. Lond. Math. Soc. (2)} {\bf 87} (2013), no. 1, 157--176.

%\bibitem[LP]{LP}
%{\sc H. Lenzing and J. A. de la Pena},
%Spectral analysis of finite dimensional algebras and singularities.
%{\it Trends in representation theory of algebras and related topics}, 541--588, EMS Ser. Congr. Rep., {\it Eur. Math. Soc., Zurich}, 2008.

%\bibitem[LS1]{LS1}
%{\sc Z. Leszczynski and A. Skowronski},
%Tame triangular matrix algebras.
%{\it Colloq. Math.} {\bf 86} (2000), no. 2, 259--303.

%\bibitem[LS2]{LS}
%{\sc Z. Leszczynski and A. Skowronski},
%Tame tensor products of algebras.
%{\it Colloq. Math.} {\bf 98} (2003), no. 1, 125--145.

%\bibitem[LVY]{LVY}
%{\sc Q. Liu, J. Vitoria and D. Yang},
%Gluing silting objects.
%{\it Nagoya Math. J.} {\bf 216} (2014), 117--151.

%\bibitem[LY1]{LY}
%{\sc Q. Liu and D. Yang},
%Blocks of group algebras are derived simple.
%{\it Math. Z.} {\bf 272} (2012), no. 3--4, 913--920.

\bibitem[LY]{LY}
{\sc Q. Liu and D. Yang},
Stratifications of algebras with two simple modules.
{\it Forum Math.} {\bf 28} (2016), no. 1, 175--188.

%\bibitem[LZ]{LZ}
%{\sc Y.-Z. Liu and Y. Zhou},
%A negative answer to complement questions for presilting complexes.
%Preprint (2023), arXiv: 2302.12502.

%\bibitem[MXZ]{MXZ}
%{\sc X. Ma, Z. Xie and T. Zhao},
%Support $\tau$-tilting modules and recollements.
%Preprint (2018), arXiv: 1801.02343.

%\bibitem[MS]{MS}
%{\sc P. Malicki and A. Skowronski},
%Cycle-finite algebras with finitely many $\tau$-rigid indecomposable modules.
%{\it Comm. Algebra} {\bf 44} (2016), no. 5, 2048--2057.

%\bibitem[M]{M}
%{\sc J.-i. Miyachi},
%Recollements and idempotent ideals.
%{\it Tsukuba J. Math.} {\bf 16} (1992), no. 2, 545--550.

%\bibitem[MW]{MW}
%{\sc K. Miyamoto and Q. Wang},
%On $\tau$-tilting finiteness of symmetric algebras of polynomial growth.
%{\it Taiwanese J. Math.} {\bf 28} (2024), no. 6, 1073--1094.

%\bibitem[Mi]{M}
%{\sc Y. Mizuno},
%Classifying $\tau$-tilting modules over preprojective algebras of Dynkin type.
%{\it Math. Z.} {\bf 277} (2014), no. 3--4, 665--690.

\bibitem[M]{Mo}
{\sc K. Mousavand},
$\tau$-tilting finiteness of biserial algebras.
Preprint (2019), arXiv: 1904.11514.

%\bibitem[N]{N}
%{\sc A. Neeman},
%The connection between the $K$-theory localization theorem of Thomason, Trobauch and Yao and the smashing subcategories of Bousfield and Ravenel.
%{\it Ann. Sci. Ecole Norm. Sup. (4)} {\bf 25} (1992), no. 5, 547--566.

\bibitem[N]{N}
{\sc J. Nehring},
Polynomial growth trivial extensions of non-simply connected algebras.
{\it Bull. Polish Acad. Sci. Math.} {\bf 36} (1988), no. 7--8, 441--445 (1989).

%\bibitem[NS]{NS}
%{\sc P. Nicolas and M. Saorin},
%Parametrizing recollement data for triangulated categories.
%{\it J. Algebra} {\bf 322} (2009), no. 4, 1220--1250.

\bibitem[O]{O}
{\sc S. Oppermann},
Quivers for silting mutation.
{\it Adv. Math.} {\bf 307} (2017), 684--714.

%\bibitem[P]{P}
%{\sc D. Pauksztello},
%Homological epimorphisms of differential graded algebras.
%{\it Comm. Algebra} {\bf 37} (2009), no. 7, 2337--2350.

%\bibitem[PSZ]{PSZ}
%{\sc D. Pauksztello, M. Saorin and A. Zvonareva},
%Contractibility of the stability manifold of silting-discrete algebras.
%{\it Forum Math.} {\bf 30} (2018), no. 5, 1255--1263.

\bibitem[Pe]{P}
{\sc J. A. de la Pena},
Algebras whose Coxeter polynomials are products of cyclotomic polynomials.
{\it Algebr. Represent. Theory} {\bf 17} (2014), no. 3, 905--930.

\bibitem[PX]{PX}
{\sc L. G. Peng and J. Xiao},
Invariability of trivial extensions of tilted algebras under stable equivalences.
{\it J. London Math. Soc. (2)}, {\bf 52} (1995), no. 1, 61--72.


\bibitem[Pl]{Pl}
{\sc P. G. Plamondon},
$\tau$-tilting finite gentle algebras are representation-finite.
{\it Pacific J. Math.} {\bf 302} (2019), no. 2, 709--716.

%\bibitem[R1]{R}
%{\sc J. Rickard},
%Morita theory for derived categories.
%{\it J. London Math. Soc. (2)} {\bf 39} (1989), no. 3, 436--456.

%\bibitem[R2]{R2}
%{\sc J. Rickard},
%Derived categories and stable equivalence.
%{\it J. Pure Appl. Algebra} {\bf 61} (1989), no. 3, 303--317.

\bibitem[R]{R}
{\sc J. Rickard},
Infinitely many algebras derived equivalent to a block.
Preprint (2013), arXiv: 1310.2403.

%\bibitem[RS]{RS}
%{\sc J. Rickard and A. Schofield},
%Cocovers and tilting modules.
%{\it Math. Proc. Cambridge Philos. Soc.} {\bf 106} (1989), no. 1, 1--5. 

%\bibitem[R]{R}
%{\sc C. M. Ringel},
%Tame algebras and integral quadratic forms.
%Lecture Notes in Mathematics, 1099.
%{\it Springer-Verlag, Berlin}, 1984.

\bibitem[Sa]{S}
{\sc M. Sato},
Periodic Coxeter matrices and their associated quadratic forms.
{\it Linear Algebra Appl.} {\bf 406 } (2005), 99--108.

\bibitem[Sc]{Sc}
{\sc S. Schroll},
Trivial extensions of gentle algebras and Brauer graph algebras.
{\it J. Algebra} {\bf 444} (2015), 183--200.

%\bibitem[S]{S}
%{\sc L. Silver},
%Noncommutative localizations and applications.
%{\it J. Algebra} {\bf 7} (1967), 44--76.

%\bibitem[SS]{SS}
%{\sc D. Simson and A. Skowronski},
%Elements of the representation theory of associative algebras.
%Vol. 2.
%Tubes and concealed algebras of Euclidean type.
%London Mathematical Society Student Texts, {\bf 71}.
%{\it Cambridge University Press, Cambridge}, 2007.

%\bibitem[S1]{S}
%{\sc A. Skowronski},
%Selfinjective algebras : finite and tame type.
%{\it Trends in representation theory of Algebras and related topics}, 69--238.
%Contemp. Math., {\bf 406}, {\it Amer. Math. Soc.}, 2006.

%\bibitem[S2]{S2}
%{\sc A. Skowronski},
%Minimal representation-infinite Artin algebras.
%{\it Math. Proc. Cambridge Philos. Soc.} {\bf 116} (1994), no. 2, 229--243.

%\bibitem[SY]{SY}
%{\sc H. Su and D. Yang},
%From simple-minded collections to silting objects via Koszul duality.
%{\it Algebr. Represent. Theory} {\bf 22} (2019), no. 1, 219--238.

\bibitem[TW1]{TW1}
{\sc H. Tachikawa and T. Wakamatsu},
Applications of reflection functors for self-injective algebras.
{\it Representation theory, I (Ottawa, Ont., 1984)}, 308--327,
Lecture Notes in Math., {\bf 1177}, {\it Springer Berlin}, 1986.

\bibitem[TW2]{TW}
{\sc H. Tachikawa and T. Wakamatsu},
Cartan matrices and Grothendieck groups of stable categories.
{\it J. Algebra} {\bf 144} (1991), no. 2, 390--398.

\bibitem[V]{V}
{\sc D. Vossieck},
The algebras with discrete derived category.
{\it J. Algebra} {\bf 243} (2001), no. 1, 168--176.

\bibitem[Wak]{Wak}
{\sc T. Wakamatsu},
Note on trivial extensions of artin algebras.
{\it Comm. Algebra} {\bf 12} (1984), no. 1--2, 33--41.

\bibitem[Wan]{W} 
{\sc Q. Wang},
On $\tau$-tilting finite simply connected algebras.
{\it Tsukuba J. Math.} {\bf 46} (2022), no. 1, 1--37.

%\bibitem[Y1]{Y1}
%{\sc D. Yang},
%The derived category of an algebra with radical square zero.
%{\it Comm. Algebra} {\bf 46} (2018), no. 2, 727--739.

%\bibitem[Y]{Y}
%{\sc D. Yang},
%Some examples of $t$-structures for finite-dimensional algebras.
%{\it J. Algebra} {\bf 560} (2020), 17--47.

%\bibitem[YY]{YY}
%{\sc L. Yao and D. Yang},
%The equivalence of two notions of discreteness of triangulated categories.
%{\it Algebr. Represent. Theory} {\bf 24} (2021), no. 5, 1295--1312.

%\bibitem[Z]{Z}
%{\sc A. Zimmermann},
%A Noether--Deuring theorem for derived categories.
%{\it Glasg. Math. J.} {\bf 54} (2012), no. 3, 647--654.

\bibitem[Z]{Z}
{\sc S. Zito},
$\tau$-tilting finite cluster-tilted algebras.
{\it Proc. Edinb. Math. Soc. (2)}, {\bf 63} (2020), no. 4, 950--955.

\end{thebibliography}
\end{document}